\documentclass{amsart}
\usepackage{mathtools}
\usepackage{amsmath,amssymb,amsthm,color,mathrsfs,array}
\usepackage{multirow}
\usepackage{verbatim}
\usepackage{url}

\newcommand{\re}{\mathbb{R}}
\newcommand{\mR}{\mathbb{R}}

\newcommand{\N}{\mathbb{N}}

\newcommand{\half}{\frac{1}{2}}
\newcommand{\lmd}{\lambda}

\newcommand{\nn}{\nonumber}
\newcommand{\eps}{\epsilon}

\def\af{\alpha}

\newcommand{\Sig}{\Sigma}

\newcommand{\reff}[1]{(\ref{#1})}

\newcommand{\mc}[1]{\mathcal{#1}}

\newcommand{\qmod}[1]{\mbox{QM}[#1]}
\newcommand{\ideal}[1]{\mbox{Ideal}[#1]}
\newcommand{\st}{\mathit{s.t.}}
\newcommand{\hm}{{(1)}}
\newcommand{\se}{ {(2)}}
\newcommand{\thi}{ {(3)} }
\newcommand{\den}{\mathit{den}}

\newcommand{\cl}[1]{\mathit{cl}(#1)}

\newcommand{\bdes}{\begin{description}}
	\newcommand{\edes}{\end{description}}
\newcommand{\bal}{\begin{align}}
\newcommand{\eal}{\end{align}}
\newcommand{\bnum}{\begin{enumerate}}
	\newcommand{\enum}{\end{enumerate}}
\newcommand{\bit}{\begin{itemize}}
	\newcommand{\eit}{\end{itemize}}
\newcommand{\bea}{\begin{eqnarray}}
\newcommand{\eea}{\end{eqnarray}}
\newcommand{\be}{\begin{equation}}
\newcommand{\ee}{\end{equation}}
\newcommand{\baray}{\begin{array}}
	\newcommand{\earay}{\end{array}}
\newcommand{\bsry}{\begin{subarray}}
	\newcommand{\esry}{\end{subarray}}
\newcommand{\bca}{\begin{cases}}
	\newcommand{\eca}{\end{cases}}
\newcommand{\bcen}{\begin{center}}
	\newcommand{\ecen}{\end{center}}
\newcommand{\bbm}{\begin{bmatrix}}
	\newcommand{\ebm}{\end{bmatrix}}
\newcommand{\bmx}{\begin{matrix}}
	\newcommand{\emx}{\end{matrix}}
\newcommand{\bpm}{\begin{pmatrix}}
	\newcommand{\epm}{\end{pmatrix}}
\newcommand{\btab}{\begin{tabular}}
	\newcommand{\etab}{\end{tabular}}

\theoremstyle{plain}
\newtheorem{theorem}{Theorem}[section]

\newtheorem{prop}[theorem]{Proposition}
\newtheorem{lem}[theorem]{Lemma}

\newtheorem{defi}[theorem]{Definition}
\newtheorem*{claim*}{Claim}
\newtheorem{thm}[theorem]{Theorem}
\theoremstyle{definition}

\newtheorem{exm}[theorem]{Example}

\newtheorem{conj}[theorem]{Conjecture}

\numberwithin{equation}{section}

\begin{document}
\title[Homogenization for polynomial optimization with unbounded sets]
{Homogenization for polynomial optimization with unbounded sets}

\author[Lei Huang]{Lei Huang}
\address{Lei Huang,
 Institute of Computational Mathematics and Scientific/Engineering Computing, Academy of Mathematics and Systems Science, Chinese Academy of Sciences,  and School of Mathematical Sciences, University of Chinese Academy of Sciences, Beijing, China, 100190.}
\email{huanglei@lsec.cc.ac.cn}

\author[Jiawang Nie]{Jiawang~Nie}
\address{Jiawang Nie,  Department of Mathematics,
University of California San Diego,
9500 Gilman Drive, La Jolla, CA, USA, 92093.}
\email{njw@math.ucsd.edu}

\author[Ya-xiang Yuan]{Ya-xiang Yuan}
\address{Ya-xiang Yuan,
 Institute of Computational Mathematics and Scientific/Engineering Computing, Academy of Mathematics and Systems Science, Chinese Academy of Sciences, Beijing, China, 100190.}
\email{yyx@lsec.cc.ac.cn}



\maketitle
\begin{abstract}
	This paper considers polynomial optimization with unbounded sets.
	We give a homogenization formulation and propose a hierarchy of
	Moment-SOS relaxations to solve it.
	Under the assumptions that the feasible set is closed at infinity and
	the ideal of homogenized  equality constraining polynomials is real radical,
	we show that  this hierarchy of Moment-SOS relaxations has finite convergence,
	if some optimality conditions
	(i.e., the linear independence constraint qualification,
	strict complementarity and second order sufficient conditions) hold
	at every minimizer, including the one at infinity.
	Moreover, we prove extended versions of
	Putinar-Vasilescu type Positivstellensatz
	for polynomials that are nonnegative on unbounded sets.
	The classical Moment-SOS hierarchy with denominators is also studied.
	In particular, we give a positive answer to a conjecture of
	Mai, Lasserre and Magron in their recent work.
\end{abstract}

\section{Introduction}
Consider the optimization problem
\be   \label{1.1}
\left\{\baray{rl}
\min & f(x) \\
\st  & c_{i}(x)=0~(i \in \mathcal{E}), \\
& c_{j}(x) \geq 0~(j \in \mathcal{I}),
\earay \right.
\ee
where $f(x),  c_i(x), c_j(x)$ are polynomials in
$x \coloneqq (x_1,\dots,x_n) \in \re^n$.
The $\mc{E}$ and $\mc{I}$ are disjoint labeling sets
for equality and inequality constraining polynomials.
Let $K$ denote the feasible set of $(\ref{1.1})$ and
let $f_{\min}$ denote the optimal value of \reff{1.1}.
This contains a broad class of important optimization problems,
such as stability numbers of graphs \cite{dKp02,mte65}
and optimization in quantum information theory \cite{FF20}.
We refer to \cite{LasBk15,LasICM,Lau09}
for related work about polynomial optimization.

The Moment-SOS hierarchy  
proposed by Lasserre~\cite{Las01} is efficient for solving
the polynomial optimization \reff{1.1}.
Under the archimedeanness for constraining polynomials
(the set $K$ must be compact for this case;
see \cite{Las01,Lau07,putinar1993positive}),
it yields a sequence of convergent lower bounds
for the minimum $f_{\min}$. Later, it was shown in \cite{nieopcd}
that the Moment-SOS hierarchy has finite convergence
if in addition the linear independence constraint qualification,
the strict complementarity and second order sufficient conditions
hold at every minimizer.\footnote{
	Throughout the paper, for convenience,
	a minimizer means a global minimizer,
	unless it is otherwise specified for the meaning.}
When the set $K$ is compact,
we refer to the work \cite{dKLLS17,dKleLau20,Las11,Schw05,SlotLau20}.
For convex polynomial optimization, the  Moment-SOS hierarchy
has finite convergence under the strict convexity
or sos-convexity conditions \cite{dKlLau11,Las09}.
When the equality constraints give a finite set,
the Moment-SOS hierarchy has finite convergence,
as shown in \cite{LLR08,Lau07,Nie13}.
More general introductions to polynomial optimization
can be found in the books and surveys
\cite{LasBk15,LasICM,Lau09,LauICM,marshall2008positive,Sch09}.

When the feasible set $K$ is unbounded (the archimedeanness fails for this case),
the classical Moment-SOS hierarchy typically does not converge.
For instance, this is the case if the objective $f$
is a nonnegative polynomial that is not a sum of squares (SOS)
and there are no constraints.
There exists work on solving polynomial optimization with unbounded sets.
Based on Karush-Kuhn-Tucker (KKT) conditions and Lagrange multipliers,
there exist tight Moment-SOS relaxations for solving \reff{1.1},
such as the work in \cite{DNP07,NDS06,nie2013exact,nie2019tight}.
In \cite{JLL14}, the authors proposed Moment-SOS relaxations
based on adding sublevel set constraints.
The resulting hierarchy of this type of relaxations
is also convergent under the archimedeanness for the new constraints.
In the above mentioned work, it is assumed that the optimal value of \reff{1.1}
is achieved. However, this is not always the case.
For instance, for $f= x_1^2 + (1-x_1x_2)^2$ and $K = \re^2$,
the optimal value $f_{\min} = 0$ is not achievable at any feasible point.
Interestingly, it was shown in \cite{ahmadi2019complexity} that
checking whether or not the optimal value is achievable is NP-hard.
When the minimum $f_{\min}$ is not achievable, we refer to the work
\cite{ha2009solving,vui2008global,schweighofer2006global}
for such polynomial optimization problems.
%
%

When the feasible set $K$ is unbounded,
Mai, Lasserre and Magron \cite{MLM21} recently  proposed
a new hierarchy of Moment-SOS relaxations. Instead of solving \reff{1.1} directly,
they considered the perturbation $f+\epsilon(1+\|x\|^2)^{d_0}$
for the objective, where $\epsilon>0$, and $d_0$ is the smallest integer such that
$2d_0 \geq \deg(f)+1$. They proved the convergence to a neighborhood of $f_{\min}$
when it is achievable.
The convergence is related to Putinar-Vasilescu's Positivstellensatz
\cite{putinar1999positive,putinar1999solving}.
The complexity for this new hierarchy
is studied in the recent work \cite{Mai2021OnTC}.
For the ideal case $\epsilon=0$, they conjectured that
their hierarchy has finite convergence
under some standard optimality conditions at each minimizer.
%

The polynomial optimization with unbounded sets
is typically hard to solve.
There are some questions of high interests.
For instance, can we get a hierarchy of Moment-SOS relaxations
that has finite convergence for almost all cases
(i.e., for generic cases)?       	
When the optimal value $f_{\min}$ is not achievable,
can we get a hierarchy of Moment-SOS relaxations
that has finite convergence?

\subsection*{Contributions}

This paper studies how to solve polynomial optimization with unbounded sets.
For a polynomial $p(x)$ of degree $\ell$, let $\tilde{p}$
denote its homogenization in  $\tilde{x} \coloneqq \left(x_0,x\right)$,
i.e., $\tilde{p}(\tilde{x}) = x_0^\ell p(x/x_0)$.
We consider the following homogenization
$\widetilde{K} \subseteq \mathbb{R}^{n+1}$ of $K$
\be \label{set:wtld:K}
\widetilde{K} \coloneqq \left\{\tilde{x} \in \mathbb{R}^{n+1}
\left| \baray{l}
\tilde{c}_{i}(\tilde{x})=0~(i \in \mathcal{E}), \\
\tilde{c}_{j}(\tilde{x}) \geq 0~(j \in \mathcal{I}), \\
\|\tilde{x}\|^2-1=0, ~ x_{0} \geq 0
\earay  \right.
\right\} .
\ee
Denote by $\tilde{c}_{\mathcal{E}}$ (resp., $\tilde{c}_{\mathcal{I}}$)
the equality (resp., inequality) constraining polynomial tuple for $\widetilde{K}$.
Let $\ideal{\tilde{c}_{\mathcal{E}}}_{2k}$ denote the $2k$th degree truncation
of the ideal of $\tilde{c}_{\mathcal{E}}$ and let $\qmod{\tilde{c}_{\mathcal{I}}}_{2k}$
denote the $2k$th degree truncation of the quadratic module of $\tilde{c}_{\mathcal{I}}$
(see Section~\ref{ssc:pre:pop} for the definition).
Suppose the degree of $f$ is $d$. For an order $k \ge \lceil \frac{d}{2}\rceil$,
we consider the $k$th order relaxation for \reff{1.1}:
\be \label{kth:SOS:hmg}
\left\{ \baray{rl}
f_k \coloneqq \max   &  \gamma \\
\st &  \tilde{f}(\tilde{x})- \gamma x_0^d \in
\ideal{\tilde{c}_{\mathcal{E}}}_{2k} + \qmod{\tilde{c}_{\mathcal{I}}}_{2k} .
\earay \right.
\ee
When the set $K$ is closed at infinity (see Definition~\ref{def:closed:inf}),
a scalar $\gamma$ is a lower bound of $f$ on $K$
if and  only if $\tilde{f}(\tilde{x})- \gamma x_0^d \geq 0 $
on $\widetilde{K}$. This is the motivation for considering
the relaxation \reff{kth:SOS:hmg}.

In this paper, we consider the case that the feasible set $K$
of \reff{1.1} is unbounded.
Our new contributions are:

\bit

\item [I.]
Our major results are to study conditions for
the hierarchy of relaxations \reff{kth:SOS:hmg}
to have finite convergence.
We prove that the finite convergence
can happen under some general assumptions,
without assuming that $f$ is positive at infinity on $K$
(for such a case, the asymptotic convergence is then guaranteed).
%
%

\item [II.]
Assume that $K$ is closed at infinity,
the ideal $\ideal{\tilde{c}_{\mathcal{E}}}$ is real radical
(see Section~\ref{ssc:pre:pop} for the definition).
When some optimality conditions (i.e.,
the linear independence constraint qualification,
strict complementarity and second order sufficient conditions)
hold at every minimizer of (\ref{1.1}),
including the one at infinity,
we show that the hierarchy of \reff{kth:SOS:hmg}
has finite convergence.
In particular, the finite convergence neither assume that
the optimal value $f_{\min}$ is achievable
nor assume that $f$ is positive at infinity on $K$.
%
%
The proof uses some classical results in \cite{nieopcd}.
To the best of the authors' knowledge,
this is the first work that proves the finite convergence for
polynomial optimization with unbounded sets.

\item [III.]
We prove that classical optimality conditions for minimizers of \reff{1.1}
are equivalent to those  for   the
homogenized optimization problem (\ref{3.5}).
This shows that the finite convergence of the hierarchy of \reff{kth:SOS:hmg}
is actually determined by optimality conditions for minimizers of \reff{1.1}.
When the ideal $\ideal{\tilde{c}_{\mathcal{E}}}$ is not real radical,
we give a new hierarchy of Moment-SOS relaxations
that also has finite convergence,
under the same assumptions on optimality conditions.
We also study genericity properties of optimality conditions.
When the polynomials are generic, we show that optimality conditions
hold at every  minimizer
and there are no minimizers at infinity.

\item [IV.] 
We give  extended versions of the
Putinar-Vasilescu's Positivstellensatz
for polynomials that are nonnegative on unbounded semialgebraic sets.
When the linear independence constraint qualification,
strict complementarity and second order sufficient conditions
hold at every minimizer of the corresponding optimization problem,
we prove that  a desired SOS type representation
with the denominator $\|x\|^{2k}$ or $(1+\|x\|^2)^k$ exists.
The Putinar-Vasilescu's Positivstellensatz gives rise to
another hierarchy of Moment-SOS relaxations for solving \reff{1.1}.
This hierarchy is the same as the one given in \cite{MLM21}
for the case $\epsilon=0$. It was conjectured in \cite{MLM21}
that this hierarchy has finite convergence
when some standard optimality conditions hold at every minimizer.
We prove that this conjecture is true,
by applying our finite convergence theory
for the hierarchy of \reff{kth:SOS:hmg}.

\eit

We would like to make the following comparisons with prior existing work
for solving polynomial optimization with unbounded sets.
\bit
\item By using KKT conditions and Jacobian representations,
a tight hierarchy of Moment-SOS relaxations is given in \cite{nie2013exact}.
By using Lagrange multiplier expressions,
an other tight hierarchy of Moment-SOS relaxations is given in \cite{nie2019tight}.
Both methods as in \cite{nie2013exact,nie2019tight}
assume that the tuple of constraining polynomials
is nonsingular and the optimal value is achievable.
In particular, the method as in \cite{nie2013exact} requires to use
defining equations for determinantal varieties. In \cite{nie2019tight},
Lagrange multiplier expressions are required.
For general polynomial constraints, it may not be convenient
to formulate the Moment-SOS relaxations as in \cite{nie2013exact,nie2019tight}.
Moreover, when the optimal value $f_{\min}$ is not achievable,
the methods as in
\cite{nie2013exact,nie2019tight} are not applicable.

\item  In the recent work by Jeyakumar et al. \cite{JLL14},
a sublevel set constraint $c- f(x) \geq 0$ is posed
as a new constraint.
When the quadratic module generated for new constraints is archimedean,
it produces an asymptotically convergent hierarchy of Moment-SOS relaxations.
When the archimedeanness fails, the asymptotic convergence is not guaranteed.
Moreover, finding such a $c$ is typically difficult for  constraint cases even if it exists and  no finite convergence results are known
for adding such new constraints.

\item When the optimal value $f_{\min}$ is not achievable,
there exist convergent hierarchies of SOS relaxations
in the work \cite{ha2009solving,vui2008global,schweighofer2006global},
under certain technical assumptions.
They are based on gradient tentacles \cite{schweighofer2006global}
or truncated tangent varieties \cite{ha2009solving,vui2008global}.
Only the asymptotic convergence is shown for
these hierarchies, under certain assumptions.
There are no finite convergence results when the
optimal value $f_{\min}$ is not achievable.

\item The recent work \cite{MLM21} of Mai, Lasserre and Magron considers
the perturbation $f+\epsilon(1+\|x\|^2)^{d_0}$,
with $2d_0 \geq \deg(f)+1$, for the objective.
When $\eps > 0$ is small and $f_{\min}$ is achievable,
they give a hierarchy of Moment-SOS relaxations
that converges to a neighborhood of the optimal value $f_{\min}$.
Finite convergence is not known for this hierarchy.
For the case $\eps = 0$, the authors conjectured that
this hierarchy has finite convergence,
under some standard optimality conditions.
We give a positive answer to this conjecture in Section~\ref{sc:denom}.

\eit

This paper is organized as follows.
Section~\ref{sc:pre} reviews some basics
about optimality conditions and polynomial optimization.
Section~\ref{sc:hmg:momsos} gives the new hierarchy
of relaxations and presents some asymptotic convergence results.
In Section~\ref{sc:ficvg}, we prove that the proposed hierarchy of relaxations
has finite convergence when some optimality conditions
hold at every minimizer, including the one at infinity.
In Section~\ref{sc:PV}, we prove extended versions of
the Putinar-Vasilescu's Positivstellensatz for polynomials
that are nonnegative on unbounded semialgebraic sets.
Section~\ref{sc:denom} studies
the Moment-SOS hierarchy with denominators.
Section~\ref{num:ex} presents some numerical experiments.
Section~\ref{sc:dis} draws  conclusions and make some discussions.

\section{Preliminaries}
\label{sc:pre}

\subsection*{Notation}

The symbol $\mathbb{N}$ (resp., $\mathbb{R}$, $\mathbb{C}$)
denotes the set of nonnegative integers (resp., real numbers, complex numbers).
For $x \coloneqq (x_1,\dots,x_n)$ and
$\alpha=(\alpha_1,\dots,\alpha_n)$, denote $x^{\alpha}  \coloneqq x_1^{\alpha_1}\cdots x_n^{\alpha_n}$ and $\lvert\alpha\rvert  \coloneqq \alpha_{1}+\cdots+\alpha_{n}$.
For a degree $d$, let $[x]_d$ denote the vector of all monomials in $x$
and whose degrees are at most $d$, ordered in the graded alphabetical ordering, i.e.,
\[
[x]_d^T \,=\, [1,  x_1, x_2, \ldots, x_1^2, x_1x_2, \ldots,
x_1^d, x_1^{d-1}x_2, \ldots, x_n^d ].
\]
Denote the power set
\[
\mathbb{N}_{d}^{n}  \coloneqq  \left\{\alpha \in \mathbb{N}^{n}: |\alpha| \le  d \right\}.
\]
Let $\mathbb{R}[x] \coloneqq \mathbb{R}[x_1,\dots,x_n]$ denote the ring of polynomials in
$x$ with real coefficients, and $\mathbb{R}[x]_d$ is the subset of polynomials in
$\re[x]$ with degrees at most $d$.  For a polynomial $p$, $\deg(p)$ denotes its total degree,
and $\tilde{p}$ denotes its homogenization, i.e.,
$\tilde{p}(\tilde{x})=x_0^{\deg(p)} p(x/x_0)$
for $\tilde{x} \coloneqq (x_0,x_1,\dots,x_n)$.
A homogeneous polynomial is said to be a form.
A form $p$ is positive definite if $p(x)>0$ for all nonzero $x \in \re^n$.
For a general polynomial $p \in \re[x]$,
$p^{(i)}$ denotes the homogeneous part of
the $i$th highest degree for $p$.
For $t \in \mathbb{R}$, $\lceil t \rceil$
denotes the smallest integer greater than or equal to $t$.
For a function $p$ in $x$,   $\nabla p$ (resp., $\nabla^2 p$) denotes its gradient (resp., Hessian)
with respect to $x$.  If $x$ is a subvector of its variables,
then $\nabla_{x}p$  denotes its gradient
with respect to $x$. If $x,y$ are two subvectors of its variables,
then  denote the  Hessian
with respect to $x=(x_{i_1},\dots,x_{i_{\ell_1}}),y=(x_{j_1},\dots,x_{\ell_2})$  by $\nabla_{x,y}^2 p$, i.e.,
\[
\nabla_{x,y}^2 p=\left(\nabla_{x_{i_t}y_{i_k}}^2p\right)_{t=1,\dots,\ell_1,k=1,\dots,\ell_2}.
\] For a matrix $A$, $A^{\mathrm{T}}$ denotes its transpose. A symmetric matrix $X \succeq 0$  if $X$ is positive semidefinite. For a vector $v$, $\|v\|$ denotes the standard Euclidean norm.
For a set $T \subseteq \mathbb{R}^n$,  let $\cl{T}$ be the closure of $T$.
A property is said to be {\it generic} for a vector space $V$
if it holds in an open dense set of $V$.

\subsection{Optimality conditions}
\label{ssc:opcd}

We review some basic theory for nonlinear programming.
Let $x^*$ be a local minimizer of $(\ref{1.1})$.
Denote the label set of active constraints at $x^*$
\be \label{nota:J(x*)}
J(x^*) \coloneqq \left\{i \in \mathcal{E} \cup \mathcal{I} :  c_{i}(x^*)=0\right\}.
\ee
The linear independence constraint qualification condition (LICQC)
is said to hold at $x^*$ if the gradient set $\{ \nabla c_{i}(x^*)\}_{i \in J(x^*)}$
is linearly independent. When the LICQC holds at $x^*$,
there exist Lagrange multipliers $\lambda_i$  $(i \in \mathcal{E} \cup \mathcal{I})$ such that
\be  \label{1.1:KKT}
\boxed{
	\begin{array}{c}
		\nabla f(x^*) = \sum_{i \in \mathcal{E} \cup \mathcal{I}} \lambda_{i} \nabla c_{i}(x^*), \\
		\lambda_{j} \geq 0~(j \in \mathcal{I}), \, \lambda_{j} c_{j}(x^*)=0~(j \in \mathcal{I}).
	\end{array}
}
\ee
The above is called the first order optimality condition (FOOC).
Moreover, if $\lambda_j + c_j(x^*) >0$ for all $j \in \mathcal{I}$,
then the strict complementarity condition (SCC) is said to hold at $x^*$.
For the above Lagrange multipliers, the Lagrange function is
\[
\mathscr{L}(x) \,  \coloneqq  \, f(x)-\sum_{i \in \mathcal{E} \cup \mathcal{I}} \lambda_{i} c_{i}(x) .
\]
The subspace of the gradients of the active constraints is denoted as
\be \label{subsp:V(x*)}
V(x^*) \,  \coloneqq  \, \mbox{span}\{\nabla c_{i}(x^*), i \in J(x^*)\}.
\ee
The orthogonal complement of $V(x^*)$ is denoted as $V(x^*)^\perp$.
Under the LICQC, the second order necessary condition (SONC) holds at $x^*$, i.e.,
\be \label{opcd:SONC}
v^{\mathrm{T}}\nabla^2 \Big( \mathscr{L}(x^*) \Big) v \geq 0, \quad
\forall ~v \in V(x^*)^{\perp}.
\ee
The second order sufficient condition (SOSC) is said to hold at $x^*$ if
\be \label{opcd:SOSC}
v^{\mathrm{T}}\nabla^2 \Big( \mathscr{L}(x^*) \Big) v > 0, \quad
\forall ~ 0 \neq v \in V(x^*)^{\perp}.
\ee
If the FOOC, SCC and SOSC hold at  $x^*$, then
$x^*$ is a strict local minimizer.
We refer to \cite{Bert97,SunYuan06} for more details
about optimality conditions.

\subsection{Some basics for polynomial optimization}
\label{ssc:pre:pop}

We review some basics in real algebraic geometry and polynomial optimization. We refer to  \cite{LasBk15,Lau09,Sch09} for more details.

A subset $I  \subseteq \mathbb{R}[x]$ is called an ideal of $\re[x]$
if $I \cdot \mathbb{R}[x] \subseteq \mathbb{R}[x]$, $I+I \subseteq I$.
For a polynomial tuple $h  \coloneqq (h_1,\dots, h_m)$,
$\ideal{h}$ denotes the ideal generated by $h$, i.e.,
\begin{equation*}
	\ideal{h} \,= \, h_1 \cdot \mathbb{R}[x]+\cdots+h_m \cdot \mathbb{R}[x].
\end{equation*}
For a degree $k$, the $k$th degree truncation of $\ideal{h}$ is
\begin{equation*}
	\ideal{h}_{k} = h_1 \cdot \mathbb{R}[x]_{k-\deg(h_1)}+\cdots+
	h_m \cdot \mathbb{R}[x]_{k-\deg(h_m)}.
\end{equation*}
Its  real variety is defined as
\begin{equation*}
	V_{\mR}(h)=\{x \in \mR^n \mid h_1(x)=\cdots=h_m(x)=0\}.
\end{equation*}
We say $\ideal{h}$ is {\it real radical} if $\ideal{h}=\ideal{V_{\mR}(h)}$, where $\ideal{V_{\mR}(h)}$ denotes the set of all polynomials vanishing on $V_{\mR}(h)$.
A polynomial $p$ is said to be a sum of squares (SOS) if
$p=p_1^2+\dots+p_t^2$ for $p_1,\dots,p_t \in \mathbb{R}[x]$.
The set of all SOS polynomials in $x$ is denoted as $\Sigma[x]$.
For an even degree $k$, denote the truncation
\[
\Sigma[x]_{k} \, \coloneqq  \, \Sigma[x] \cap  \mathbb{R}[x]_{k} .
\]
For a polynomial tuple $g=(g_1,\dots,g_{\ell})$,
the  quadratic module generated by $g$ is
\be
\qmod{g} \,  \coloneqq  \,  \Sigma[x]+ g_1 \cdot \Sigma[x]+\cdots+ g_{\ell} \cdot \Sigma[x].
\ee
Similarly, for an even degree $k$, the $k$th degree truncation of $\qmod{g}$ is
\be
\qmod{g}_{k} = \Sigma[x]_{k}+ g_1 \cdot \Sigma[x]_{k-2  \lceil \deg(g_1)/2 \rceil}+\cdots+ g_{\ell} \cdot \Sigma[x]_{k-2\lceil \deg(g_{\ell}) /2\rceil}.
\ee

The sum $\ideal{h}+\qmod{g}$ is said to be archimedean if there exists
$R >0$ such that $R-\|x\|^2 \in \ideal{h}+\qmod{g}$.
If it is archimedean, then the set
\[
S \coloneqq \left\{x \in \mathbb{R}^{n} \mid h(x)=0, g(x) \geq 0\right\}
\]
must be compact. Clearly, if $p \in \ideal{h}+\qmod{g}$, then $p \ge 0$ on $S$
while the converse is not always true. However,
if $p$ is positive on $S$ and $\ideal{h}+\qmod{g}$ is archimedean, we have $p \in \ideal{h}+\qmod{g}$.
This conclusion is referred to as Putinar's Positivstellensatz.

\begin{thm}[\cite{putinar1993positive}] \label{thm2.1}
	Suppose  $\ideal{h}+\qmod{g}$ is  archimedean.
	If a polynomial $p>0$ on  $S$, then $p \in \ideal{h}+\qmod{g}$.
\end{thm}

For an integer $k>0$, let $\re^{ \N^n_{2k} }$ denote the space of all real vectors
that are labeled by $\af \in \N^n_{2k}$. Each
$y \in \re^{ \N^n_{2k} }$ is labeled as
\[
y \, = \, (y_\af)_{ \af \in \N^n_{2k} }.
\]
Such $y$ is called a
{\it truncated multi-sequence} (tms) of degree $2k$.
For a polynomial $p  = \sum_{ |\af| \le 2k } p_\af x^\af \in  \re[x]_{2k}$,
define the bilinear operation
\be \label{<f,y>}
\langle p, y \rangle \,= \, \sum_{ |\af| \le 2k } p_\af y_\af.
\ee
For given $p$, $\langle p, y \rangle$ is a linear function in $y$.
For  the degree
$t  \coloneqq  k - \lceil \deg(p)/2 \rceil$,
the localizing matrix $L_{p}^{(k)}[y]$
is a symmetric linear matrix function in $y$ such that
\be \label{df:Lf[y]}
q^T \Big( L_{p}^{(k)}[y] \Big) q \, = \,
\langle p (q^T[x]_t)^2, y \rangle
\ee
for all real vector $q \in \re^{ \N^n_t }$.
In particular, if $p=1$ is the constant one polynomial,
then $L_{1}^{(k)}[y]$ is called a moment matrix, for which we denote as
\[
M_k[y] \,  \coloneqq  \, L_{1}^{(k)}[y].
\]
%
%
Localizing matrices are quite useful for solving
polynomial and tensor optimization problems
\cite{FNZ18,HenLas05,LasBk15,Lau09,nieloc,NYZ18}.

\section{Moment-SOS relaxations for the homogenization}
\label{sc:hmg:momsos}

In this section, we give a hierarchy of Moment-SOS relaxations
based on the homogenization of \reff{1.1}
and study its properties.
Let $\tilde{x} \coloneqq (x_0,x)$.
For a polynomial $p \in \re[x]$, let $\tilde{p}(\tilde{x})$
denote the homogenization of $p$.
For the feasible set $K$ as in $(\ref{1.1})$, denote the sets
\be \label{sets:KKK}
\boxed{
	\baray{rcl}
	\widetilde{K}^h & \coloneqq & \left\{  \tilde{x}
	\left| \baray{l}
	\tilde{c}_{i}(\tilde{x})=0~(i \in \mathcal{E}), \\
	\tilde{c}_{j}(\tilde{x}) \geq 0~(j \in \mathcal{I}), \\
	x_{0} \geq 0
	\earay \right. \right\}, \\
	\widetilde{K}^c & \coloneqq &  \widetilde{K}^h \cap \{ x_0 > 0  \}, \\
	\widetilde{K}   & \coloneqq &   \widetilde{K}^h \cap \{ x_0^2 + x^Tx = 1 \} .
	\earay
}
\ee
Define the perspective projection map $\varphi$ as
\be \label{map:varphi}
\varphi:\left\{\left(x_{0}, x\right) \in \mathbb{R}^{n+1} \mid x_{0}>0,\,
x_0^2 + x^Tx = 1 \right\}
\rightarrow \mathbb{R}^{n}, \quad
\left(x_{0}, x\right) \mapsto  \frac{x}{x_{0}} .
\ee
The map $\varphi$
gives a one-to-one correspondence between
$ \widetilde{K} \cap \{ x_0 > 0\}$ and $K$.

The set $\widetilde{K}^c$ only depends on the geometry of $K$,
while $\widetilde{K}^h$ and $\widetilde{K}$ depend on
the description polynomials for $K$.
Throughout the paper, we assume their description polynomials
are  as in \reff{1.1}.
The following is a useful definition for homogenization.

\begin{defi}[\cite{njwdis}] \label{def:closed:inf}
	The set $K$ is closed at infinity ($\infty$)
	if $\cl{\widetilde{K}^c}=\widetilde{K}^h$.
\end{defi}

The above definition was introduced in \cite{njwdis}
for studying the boundary of the cone of polynomials nonnegative on $K$.
The following is a basic property for closedness at $\infty$.

\begin{lem}	\label{lemma3.1}
	For $f \in \re[x]$, we have $f \geq 0$ on $K$ if and only if
	$\tilde{f} \geq 0$ on $cl(\widetilde{K}^c)$.
	Moreover, when $K$ is closed at $\infty$, $f \geq 0$ on $K$ if and only if
	$\tilde{f} \geq 0$ on $\widetilde{K}$.
\end{lem}

Being closed at $\infty$ is a generic property
for semialgebraic sets (see \cite{guo2014minimizing}).
We remark that the closedness at $\infty$
may depend on the choice of description polynomials for $K$.
For instance, the following two sets
\[
S_1=\{(x_1,x_2)\in \mR^2: x_1-x_2^2 \geq 0  \}, \quad
S_2=\{(x_1,x_2)\in \mR^2: x_1-x_2^2 \geq 0 , x_1 \geq 0 \}
\]
are the same. However, for their description polynomials,
the set $S_2$ is closed at $\infty$ while $S_1$ is not,
since $(0, -1,0) \in \widetilde{S}_1^h \backslash \cl{\widetilde{S}_1^c}$.
Throughout the paper, when we mention $K$ is closed at $\infty$,
it means that their description polynomials as in \reff{1.1} are used.

In view of the minimum value,
the optimization $(\ref{1.1})$  is equivalent to
\begin{equation}     \label{3.1}
	\left\{ \baray{rl}
	\max  &\gamma \\
	\st &f(x)-\gamma \geq 0 \quad \mbox{on} \quad  K.
	\earay \right.
\end{equation}
Let $d  \coloneqq \deg(f)$. When $K$ is closed at $\infty$,
one has $f-\gamma \ge 0$ on $K$ if and only if
$\tilde{f}(\tilde{x}) - \gamma x_0^d \ge 0$ on $\widetilde{K}$.
Therefore, \reff{3.1} is equivalent to
\begin{equation}  \label{3.2}
	\left\{ \baray{rl}
	\max  &\gamma \\
	\st & \tilde{f}(\tilde{x}) - \gamma x_0^d \geq 0 \
	\quad  \mbox{on} \quad \widetilde{K}.
	\earay \right.
\end{equation}
For convenience of notation, denote the polynomial sets
\be \label{sets:tld:ceqcin}
\tilde{c}_{\mathcal{E}}  \coloneqq  \big\{ \tilde{c_i}(\tilde{x}) \big\}_{i\in \mc{E} }
\cup \{ \|\tilde{x}\|^2-1 \}, \quad
\tilde{c}_{\mathcal{I}}  \coloneqq  \big\{ \tilde{c_j}(\tilde{x}) \big\}_{j \in \mc{I} }
\cup \{ x_0 \}.
\ee
Note that $\widetilde{K}$ is compact and
$\ideal{\tilde{c}_{\mathcal{E}}}+\qmod{\tilde{c}_{\mathcal{I}}}$ is archimedean for all $K$.

We apply Moment-SOS relaxations to solve \reff{3.2}.
For an order $k \geq \lceil \frac{d}{2}\rceil$, the $k$th order SOS relaxation
for \reff{3.2} is
\begin{equation}  \label{3.3}
	\left\{ \baray{rl}
	\max &   \gamma \\
	\st &  \tilde{f}(\tilde{x})- \gamma x_0^d \in
	\ideal{\tilde{c}_{\mathcal{E}}}_{2k} +\qmod{\tilde{c}_{\mathcal{I}}}_{2k}.
	\earay \right.
\end{equation}
The dual optimization of $(\ref{3.3})$ is the $k$th order moment relaxation
\begin{equation}      \label{d3.3}
	\left\{ \baray{rl}
	\min  &  \langle \tilde{f}, y \rangle  \\
	\st  & L_{p}^{(k)}[y] =0\, ~(p \in \tilde{c}_{\mathcal{E}}),\\
	& L_{q}^{(k)}[y] \succeq 0\, ~(q \in \tilde{c}_{\mathcal{I}}), \\
	&  M_k[y] \succeq 0, \\
	& \langle x_0^d, y \rangle = 1, \,
	y \in \re^{ \N^{n+1}_{2k} }.
	\earay \right.
\end{equation}
Let $f_k$ and $f^{\prime}_k$ denote the optimal values of
\reff{3.3}, \reff{d3.3} respectively.
As $k$ goes to infinity, the sequence of the relaxations
\reff{3.3}-\reff{d3.3} is called the homogenized hierarchy
of Moment-SOS relaxations for solving \reff{1.1}.

\subsection{Basic properties}
\label{ssc:pro:hmg}

When $K$ is closed at $\infty$,
the optimal values of \reff{3.1} and \reff{3.2} are the same.
This is a basic property of homogenization.
The following is implied by Lemma~\ref{lemma3.1}.

\begin{prop} \label{prop3.3}
	If $K$ is closed at $\infty$, then the optimal values of
	\reff{3.1} and \reff{3.2} are the same.
\end{prop}

When the degrees of $f$, $c_j$ ($j \in \mathcal{I}$) are all even,
the constraint $x_0\geq 0$ is redundant for  homogenization.
Consider the optimization problem
\be  \label{lo3.5}
\left\{ \baray{rl}
\max  &\gamma \\
\st & \tilde{f}(\tilde{x}) - \gamma x_0^d \geq 0 \
\quad  \mbox{on} \quad \widetilde{K}^{e},
\earay \right.
\ee
where the set $\widetilde{K}^{e}$ is
\be
\baray{rcl}
\widetilde{K}^{e}  \coloneqq & \left\{  \tilde{x} \in \mathbb{R}^{n+1}
\left| \baray{l}
\tilde{c}_{i}(\tilde{x})=0~(i \in \mathcal{E}), \\
\tilde{c}_{j}(\tilde{x}) \geq 0~(j \in \mathcal{I}), \\
x_0^2 + x^Tx = 1
\earay \right. \right\}. \\
\earay
\ee
Suppose $K$ is closed at $\infty$ and $f\geq 0$ on $K$.
By Lemma~\ref{lemma3.1}, we have $\tilde{f}(\tilde{x})\geq0$ for $\tilde{x}=(x_0,x) \in \widetilde{K}^{e}$
with $x_0\geq 0$. For $\tilde{x} \in \widetilde{K}^{e}$ with $x_0< 0$,
we have $\tilde{c}_i(\tilde{x})=(x_0)^{\deg(c_i)} c_i(\frac{x}{x_0})$
for each $i \in \mathcal{E} \cup \mathcal{I}$.
This implies that  $\frac{x}{x_0} \in K$ and
$\tilde{f}(\tilde{x})=x_0^{d}f(\frac{x}{x_0})\geq0$ for all $x_0 < 0$,
when the degrees of $f$, $c_j$ ($j \in \mathcal{I}$) are even.
Thus, the following holds.

\begin{prop} \label{prop3.4}
	If the set $K$ is closed at $\infty$ and the degrees of
	$f$, $c_j$ ($j \in \mathcal{I}$) are even, then the optimal values of
	\reff{3.1} and \reff{lo3.5} are the same.
\end{prop}

We remark that if one of $f$ and $c_j$ ($j \in \mathcal{I}$)
has an odd degree, then \reff{3.1} and \reff{lo3.5}
may not have the same optimal value.
For instance, consider the optimization
\be  \nn
\min  \quad x  \quad
\st  \quad 1-x^2 \geq 0, \, x \in \re.
\ee
The minimum value $f_{\min}=-1$ and
$\widetilde{K}^{e} = \{ x_0^2 \ge x^2,\, x_0^2 + x^2 = 1 \}$.
However, the optimal value of \reff{lo3.5}
is $-\infty$, since there is no scalar $\gamma$
such that $x - \gamma x_0 \ge 0$ on $\widetilde{K}^{e}$.

Suppose the optimal value $f_{\min} > -\infty$.
A point $x^* \in K$ is a  minimizer of $(\ref{1.1})$
if and only if $f(x^*) -f_{\min}=0$, which is equivalent to
\be \label{hom:eq}
\tilde{f}(\tilde{x}^*)- f_{\min} \cdot (\tilde{x}^*_0)^d=0
\ee
for the point
$
\tilde{x}^* \coloneqq (1+\|x^*\|^2)^{-\half} (1, x^*) .
$
By the perspective projection $\varphi$ as in \reff{map:varphi},
it is easy to see there exists a one-to-one correspondence between
minimizers of $(\ref{1.1})$ and feasible points of \reff{hom:eq} with $\tilde{x}=(x_0,x) \in \widetilde{K}$, $x_0>0$.
However, the system \reff{hom:eq} may have feasible points of the form
$\tilde{x}^*=(0, u)$. For such a case, it does not give a minimizer of \reff{1.1}.

\begin{defi} \rm
	A point $x^*$ is said to be a minimizer at infinity
	for \reff{1.1} if $\tilde{x}^*=(0, x^*)$
	satisfies \reff{hom:eq} and $\tilde{x}^* \in  \widetilde{K}$.
\end{defi}

Recall that for a polynomial $p$, $p^{(1)}$ denotes the homogeneous part of
the highest degree for $p$. For convenience, denote the set
\be \label{set:K:hom}
K^\hm  \coloneqq
\left\{x \in \mathbb{R}^n
\left| \baray{l}
c^\hm_{i}(x)=0~(i \in \mathcal{E}), \\
c^\hm_{j}(x) \geq 0~(j \in \mathcal{I}), \\
\|x\|^2-1=0
\earay \right.
\right \}.
\ee

The following is a basic property of homogenization.

\begin{theorem}  \label{thm:prop:inf}
	Suppose  the minimum value $f_{\min} > -\infty$.
	Then, we have:
	\bit
	
	\item [(i)] A point $x^*$ is a minimizer at $\infty$ for \reff{1.1}
	if and only if $f^\hm(x^*)=0$ and $x^* \in K^\hm$.
	
	\item [(ii)] If the minimum value $f_{\min}$ is not achievable for \reff{1.1},
	then \reff{1.1} has a minimizer at $\infty$.
	
	\item [(iii)] Suppose $K$ is closed at $ \infty$. Then
	the form $f^\hm \geq 0$ on $K^\hm$.
	In particular, if  $f^\hm > 0$ on $K^\hm$,
	then there are no minimizers at $\infty$.
	
	\eit
\end{theorem}
\begin{proof}
	(i) For $\tilde{x}^*  \coloneqq (0,x^*)$,
	note that $\tilde{c}_i(\tilde{x}^*)=c_i^\hm(x^*)$ for all $i$.
	So, $\tilde{x}^* \in \widetilde{K}$ if and only if $x^* \in K^\hm$.
	Since  $f_{\min} > -\infty$, one can see that
	\[
	0 = \tilde{f}(\tilde{x}^*)- f_{\min} \cdot (\tilde{x}^*_0)^d
	\, = \, f^\hm(x^*) .
	\]
	So the conclusion is true.
	
	(ii) Since the optimal value is not achievable, the set
	$K$ is unbounded and there is a sequence
	$\{ x^{(k)} \}_{k=1}^\infty \subseteq K$ such that
	$f(x^{(k)}) \rightarrow f_{\min}$
	and $\|x^{(k)}\| \rightarrow \infty$.
	Let $y^{(k)}  \coloneqq  x^{(k)}/\|x^{(k)}\|$ be the normalization.
	Without loss of generality, one can further assume $y^{(k)} \rightarrow y^*$. Clearly, $\|y^*\|=1$, and
	\[
	f^\hm(y^*) =\lim\limits_{k \rightarrow \infty}f^\hm(y^{(k)})=
	\lim\limits_{k \rightarrow \infty}  f(x^{(k)})/\|x^{(k)}\|^d = 0.
	\]
	Similarly, one can show that
	\[
	c_i^\hm(y^*)
	=0~ (i \in \mc{E}), \quad
	c_i^\hm(y^*)
	\geq 0\,~ (j \in \mc{I}).
	\]
	So $y^* \in K^\hm$ and $y^*$ is a minimizer at $\infty$,
	by the conclusion in item (i).
	
	(iii)
	If $K^\hm = \emptyset$, the conclusion is clearly true.
	We consider the case that $K^\hm \neq \emptyset$. For each $u \in K^\hm$,
	we have $\tilde{u}  \coloneqq (0,u) \in \widetilde{K}^h$.
	Since $K$ is closed at $\infty$, there exists a sequence  of
	$\tilde{u}^{(k)} =(u_0^{(k)}, u^{(k)}) \in \widetilde{K}^c$
	such that $\tilde{u}^{(k)} \rightarrow \tilde{u}$ and each $u_0^{(k)}>0$.
	Note that $u^{(k)}/u_0^{(k)} \in K$ and
	\[
	\tilde{f}(\tilde{u}^{(k)}) - f_{\min} \cdot (u_0^{(k)})^d =
	(u_0^{(k)})^d(f (u^{(k)}/u_0^{(k)})-  f_{\min}) \geq 0.
	\]
	Letting $k \rightarrow \infty$, we get $f^\hm(u) \geq 0$
	for every $u \in K^\hm$, hence $f^\hm \ge 0$ on $K^\hm$.
	
	When the form  $f^\hm > 0$ on $K^\hm$,
	there are no minimizers at $\infty$.
	This is implied by the item (i).
\end{proof}

\subsection{Asymptotic convergence}
\label{ssc:asycvg}

The asymptotic convergence of the Moment-SOS hierarchy of
\reff{3.3}-\reff{d3.3} can be shown under the following condition.
%
%

\begin{defi} \label{def:pos:inf}
	A polynomial $p$ is said to be positive at $\infty$ on $K$
	if its highest degree homogeneous part $p^\hm >0$ on $K^\hm$.
\end{defi}

This condition has appeared in \cite{DPP,marshall2006}
to study different properties of nonnegative polynomials on unbounded sets.
If the objective $f$ is positive at $\infty$,
then $f$ is coercive on $K$, i.e., for every value $\vartheta \in \re$,
the sublevel set $\{ x \in K: \, f(x) \le \vartheta \}$
is compact. This is shown in the following lemma.

\begin{lem}  \label{4.1}
	If $f$ is positive at $\infty$ on $K$, then $f$ is coercive on $K$.	
\end{lem}
\begin{proof}
	When $K$ is compact, the conclusion is clearly true.
	Consider the case that $K$ is unbounded.
	Suppose otherwise $f$ was not coercive on $K$,
	then there exist a value $\vartheta$ and  a sequence
	$\{ u^{(k)} \}_{k=1}^\infty \subseteq K$ such that
	\[
	\| u^{(k)} \|  \rightarrow \infty, \quad
	f(u^{(k)}) \leq \vartheta \,\,\, \mbox{for all} \, k .
	\]
	Let $\tilde{u}^{(k)}  \coloneqq  (1,u^{(k)})/\sqrt{1+\|u_k\|^2}$,
	then $\tilde{u}^{(k)} \in \widetilde{K}$ and $\|\tilde{u}^{(k)}\|=1$
	for all $k$. Without loss of generality, we can further assume that
	$\tilde{u}^{(k)} \rightarrow \tilde{u}^* \coloneqq  (0,u^*) \in \widetilde{K}$
	as $k \rightarrow \infty$. Note that $\|u^*\|=1$,  and
	\[
	f^\hm(u^*)=\tilde{f}(\tilde{u}^*)=
	\lim\limits_{k \rightarrow \infty} \tilde{f}(\tilde{u}^{(k)})=
	\lim\limits_{k \rightarrow \infty}
	\frac{f(u^{(k)})}{\sqrt{1+\|u^{(k)}\|^2}^d} \le 0.
	\]
	One can similarly show that
	$c_i^\hm(u^*)= 0$ for every $i \in \mathcal{E}$ and
	$c_j^\hm(u^*) \geq 0$ for every $j \in \mathcal{I}$.
	Hence, we get $u^* \in K^\hm$ and $f^\hm(u^*) \leq 0$,
	a contradiction to the positivity of $f$ at $\infty$.
	So $f$ must be coercive on $K$.
\end{proof}

We remark that if $f$ is coercive on $K$,
it may not be positive at $\infty$ on $K$.
For instance, the polynomial $f=x_1^4 +x_2^2$
is not positive at $\infty$
for $K = \re^2$, but it is coercive.
It is typically a hard question to check coercivity,
as shown in \cite{ahmadi2019complexity}.
Coercivity of polynomials is also studied in
\cite{Tom2015Coercive,JLL14}.
The coercivity is sufficient but not necessary
for the optimal value to be achievable.

When $f$ is positive at $\infty$ on $K$,
the hierarchy of relaxations \reff{3.3}
has asymptotic convergence.
This is shown as follows.

\begin{thm}
	If $f$ is positive at $\infty$ on $K$,
	then the optimization \reff{1.1} achieves the minimum value
	and $f_k \rightarrow f_{\min}$ as $k \to \infty$.
\end{thm}
\begin{proof}
	Since $f$ is positive at $\infty$ on $K$,
	Lemma~\ref{4.1} implies that $f$ is coercive on $K$,
	so \reff{1.1} must achieve its minimum value $f_{\min} > -\infty$
	and it has minimizers.
	For every scalar $\gamma < f_{\min}$, we show that
	$\tilde{f}(\tilde{u})- \gamma u_0^d >0$ for all
	$\tilde{u}  \coloneqq  (u_0, u)  \in \widetilde{K}$.
	If $u_0 = 0$, then $u \in K^\hm$ and $\tilde{f}(\tilde{u}) = f^\hm(u) > 0$,
	since $f$ is positive at $\infty$ on $K$.
	If $u_0 > 0$, then $u/u_0 \in K$ and
	\[
	\tilde{f}(\tilde{u}) - \gamma (u_0)^d = u_0^{d}(f(u/u_0 )-\gamma) > 0.
	\]
	Note that $\ideal{\tilde{c}_{\mathcal{E}}}+\qmod{\tilde{c}_{\mathcal{I}}}$ is archimedean,
	due to the sphere constraint $x_0^2 + x^Tx = 1$.
	By Theorem \ref{thm2.1}, we have
	$\tilde{f}(\tilde{x})- \gamma x_0^d \in
	\ideal{\tilde{c}_{\mathcal{E}}}+\qmod{\tilde{c}_{\mathcal{I}}}$.
	Thus, when $k$ is sufficiently large, we get
	$\tilde{f}(\tilde{x})- \gamma x_0^d \in
	\ideal{\tilde{c}_{\mathcal{E}}}_{2k}+\qmod{\tilde{c}_{\mathcal{I}}}_{2k}$.
	This is true for every $\gamma < f_{\min}$.
	On the another hand, if $\gamma$ is feasible for \reff{3.3},
	we must have $\tilde{f}(\tilde{x})- \gamma x_0^d \geq 0$ on $\widetilde{K}$,
	which implies that $f-\gamma \geq 0$ on $K$ and hence $\gamma \leq f_{\min}$.
	This shows that $f_k \rightarrow f_{\min}$ as $k \to \infty$.
\end{proof}

The following is an example for the hierarchy of \reff{3.3}-\reff{d3.3}.
\begin{exm} \label{example3.3}
	Consider the optimization problem
	\[
	\left\{ \baray{rl}
	\min &  x_1+x_2  \\
	\st & x_1^3+x_2 +1  \geq 0, \, x_2^3-x_1+1 \geq 0.
	\earay \right.
	\]
	The feasible set $K$ is unbounded. One can check that the minimum value
	and the unique minimizer are respectively
	\[
	f_{\min}=-1-\frac{2\sqrt{3}}{9}, \quad
	x^*=(-\frac{\sqrt{3}}{3},-1+\frac{\sqrt{3}}{9}).
	\]
	Note that $f^\hm = x_1+x_2$ and
	\[
	K^\hm \,=\, \{x_1^3 \geq 0, x_2^3 \geq 0, x_1^2+x_2^2=1 \} .
	\]
	The form $f^\hm$ is positive on $K^\hm$.
	The hierarchy of \reff{3.3}-\reff{d3.3} has the asymptotic convergence.
	Interestingly, it also has finite convergence.
	This can be implied by Theorem~\ref{thm3.12}.
\end{exm}

\section{Optimality conditions and finite convergence }
\label{sc:ficvg}

Optimality conditions are closely related to finite convergence of
the classical Moment-SOS hierarchy in \cite{Las01}.
Under the archimedeanness for constraining polynomials,
the Moment-SOS hierarchy has finite convergence
when the linear independence constraint qualification,
strict complementarity and second order sufficient conditions
hold at every  minimizer. This is shown in \cite{nieopcd}.
When the feasible set $K$ is unbounded, the above conclusion may not hold.
However, we can prove similar conclusions for
the homogenized hierarchy of relaxations \reff{3.3}-\reff{d3.3}.

Recall that $f_k$ is the optimal value of the relaxation \reff{3.3}
for the relaxation order $k$. The hierarchy of \reff{3.3}-\reff{d3.3}
is said to have finite convergence or be tight if
$f_k = f_{\min}$ for all $k$ big enough.
To guarantee the finite convergence,
we assume that $K$ is closed at $\infty$,
$\ideal{ \tilde{c}_{\mathcal{E}} }$ is real radical,
the LICQC, SCC and SOSC hold at every minimizer,
including the one at infinity.
However, the set $K$ is not assumed to be bounded
and we do not assume the optimal value of \reff{1.1} is achievable.

We consider the homogenized optimization problem
(note $\tilde{x} = (x_0, x)$)
\be   \label{3.5}
\left\{ \baray{rl}
\min  &  F(\tilde{x}) \coloneqq \tilde{f}(\tilde{x}) - f_{\min} \cdot x_0^d  \\
\st &   \tilde{c}_{i}(\tilde{x})=0~(i \in \mathcal{E}), \\
&   x_0^2 + \| x \|^2-1=0, \\
&  \tilde{c}_{j}(\tilde{x}) \geq 0~(j \in \mathcal{I}),\\
&   x_{0} \geq 0 .
\earay \right.
\ee
Since there is the sphere constraint,
the above optimization must have minimizers.

\begin{lem}
	\label{hom:opt:zero}
	If $K$ is closed at $\infty$ and $f_{\min}>-\infty$,	
	the minimum value of \reff{3.5} is 0.
\end{lem}
\begin{proof}
	Since $K$ is closed at $ \infty$, we know $F(\tilde{x}) \geq 0$
	on $\widetilde{K}$, by Lemma \ref{lemma3.1}.
	If $f_{\min}$ is achievable at a minimizer $x^*$ of $\reff{1.1}$,
	then $\tilde{x}^* \coloneqq (1+\|x^*\|^2)^{-\half} (1,x^*) \in \widetilde{K}$
	and $F(\tilde{x}^*)=0$.
	If $f_{\min}$ is not achievable, then \reff{1.1}
	has a minimizer at $\infty$, say $u^*$, by Theorem \ref{thm:prop:inf}.
	For the point $\tilde{u}^* \coloneqq (0,u^*)$,
	we have $F(\tilde{u}^*)=0$ and $\tilde{u}^* \in \widetilde{K}$.
	So the minimum value of \reff{3.5} is 0.
\end{proof}

Since the minimum value is $0$, a point $\tilde{u}$ is a minimizer of \reff{3.5}
if and only if $F(\tilde{u})=0$ and $\tilde{u} \in \widetilde{K}$.
Suppose $\tilde{u}=(u_0,u)$ is a minimizer of \reff{3.5}.
If $u_0 >0$, then $u/u_0$ is a minimizer of \reff{1.1}.
If $u_0 =0$, then $u$ is a minimizer at $\infty$ for \reff{1.1}.
For the case $u_0 >0$, we call $\tilde{u}$ a regular minimizer of \reff{3.5}.
For the case $u_0 =0$, we call $\tilde{u}$ a minimizer at infinity of \reff{3.5}.
Recall the optimality conditions LICQC, SCC and SOSC
as in Subsection~\ref{ssc:opcd}.
Note that $\widetilde{K}$ is always compact.
The following is the finite convergence result
based on optimality conditions
for the homogenized optimization problem \reff{3.5}, which follows from \cite{nieopcd}.

\begin{lem}  \label{hom:opt}
	Assume $K$ is closed at $\infty$ and $\ideal{\tilde{c}_{\mathcal{E}}}$
	is real radical. If the LICQC, SCC and SOSC hold
	at every minimizer of \reff{3.5}, then the hierarchy of relaxations
	\reff{3.3}-\reff{d3.3} is tight, i.e., there exists $k_0 \in \N$ such that
	\[
	f_k=f^{\prime}_k=f_{\min} \quad \mbox{for all}\, \, k \ge k_0.
	\]
\end{lem}
\begin{proof}
	Note that $\ideal{\tilde{c}_{\mathcal{E}}} +\qmod{\tilde{c}_{\mathcal{I}}}$ is archimedean.
	By Theorem~1.1 of \cite{nieopcd},
	there exists $\sigma \in \qmod{\tilde{c}_{\mathcal{I}}}$ such that
	\begin{equation*}
		\tilde{f}-f_{\min}x_0^d \equiv \sigma ~ \bmod ~
		\ideal{V_{\mathbb{R}}(\tilde{c}_{\mathcal{E}})},
	\end{equation*}
	Since the ideal $\ideal{ \tilde{c}_{\mathcal{E}} }$ is real radical,
	there exists $\phi \in \ideal{ \tilde{c}_{\mathcal{E}} }$ such that
	\[
	\tilde{f}-f_{\min}x_0^d = \phi + \sigma.
	\]
	This implies that $f_k=f^{\prime}_k=f_{\min}$ for all $k$ big enough.
\end{proof}

In Lemma \ref{hom:opt}, optimality conditions are stated for
the homogenized optimization problem \reff{3.5}.
Interestingly, the optimality conditions at regular minimizers for \reff{3.5}
are equivalent to those for \reff{1.1}.
The equivalence will be shown in the following subsections
(see Theorem~\ref{lemma3.11}).
Therefore, the finite convergence result can be stated
under the optimality conditions of \reff{1.1}.
For a minimizer  at infinity $x^*$,
we say the LICQC, SCC and SOSC hold at $x^*$
if they hold for \reff{3.5} at $(0,x^*)$.
The following is the main result about the finite convergence.

\begin{theorem}  \label{thm3.12}
	Assume $K$ is closed at $\infty$ and the ideal $\ideal{\tilde{c}_{\mathcal{E}}}$
	is real radical. If the LICQC, SCC and SOSC hold at every  minimizer of \reff{1.1},
	including the one at infinity, then the hierarchy of relaxations
	\reff{3.3}-\reff{d3.3} is tight, i.e.,
	$f_k=f^{\prime}_k=f_{min}$ for all $k$ big enough.
\end{theorem}
\begin{proof}
	Since $K$ is closed at $ \infty$, we know $\tilde{f}-f_{\min} x_0^d \geq 0$ on $\widetilde{K}$, by Lemma \ref{lemma3.1}.
	Let $\tilde{x}^*=(x^*_0,x^*) \in \widetilde{K}$ be a  minimizer of \reff{3.5}.
	If $x^*_0 > 0$, then $\frac{x^*}{x^*_0}$ is a  minimizer of $(\ref{1.1})$.
	Since the LICQC, SCC and SOSC hold at $\frac{x^*}{x^*_0}$,
	these optimality conditions also hold at $\tilde{x}^*$ for \reff{3.5},
	by Theorem \ref{lemma3.11}. If $x_0^*=0$,
	then $x^*$ is a  minimizer at infinity.
	By the assumptions, these optimality conditions also hold at $\tilde{x}^*$.
	Hence, the LICQC, SCC and SOSC hold at every  minimizer of $ (\ref{3.5})$.
	By Lemma \ref{hom:opt},
	we have $f_k=f^{\prime}_k=f_{\min}$ for all $k$ big enough.
\end{proof}

In Theorem~\ref{thm3.12}, the minimum value
$f_{\min}$ is not assumed to be achievable for \reff{1.1}.
When there are no equality constraints,
the ideal $\ideal{ \tilde{c}_{\mathcal{E}} } = \ideal{ \|\tilde{x}\|^2-1 }$
is real radical, regardless of inequality constraints.

For an exposition for Theorem~\ref{thm3.12}, we consider
Example~\ref{example3.3}. At the unique minimizer $
x^*=(-\frac{\sqrt{3}}{3},-1+\frac{\sqrt{3}}{9})
$, the first constraint $x_1^3+x_2 +1 \geq 0$ is active
while the second one is not.
Hence, we have that the Lagrange multipliers $\lmd_1 = 1$, $\lmd_2 = 0$.
The Hessian of the Lagrange function at $x^*$ is
$
\bbm 2\sqrt{3}   & 0 \\ 0 & 0 \ebm.
$
One can see that the LICQC, SCC and SOSC all hold  at $x^*$.
One can check that there are no minimizers at infinity and
$\ideal{ \tilde{c}_{\mathcal{E}}} = \ideal{ \| \tilde{x}\|^2 -1 }$
is real radical. By Theorem~\ref{thm3.12},
the hierarchy of \reff{3.3}-\reff{d3.3}
is tight for this optimization problem.
In fact, we have $f_k = f_{\min}$ for all $k\geq 3$.

In the following, we investigate when the
optimality conditions hold for \reff{3.5}.

\subsection{Optimality conditions for regular minimizers}
\label{ssc:preserve}

An important property of homogenization is that
it preserves optimality conditions for regular minimizers.

\begin{theorem} \label{lemma3.11}
	Suppose  $K$ is closed at $\infty$.
	Then, a point $x^* \in K$ is a  minimizer of \reff{1.1}
	if and only if $\tilde{x}^*  \coloneqq  (1,x^*)/\sqrt{1+\|x^*\|^2}$
	is a  minimizer of \reff{3.5}. Moreover, we have:
	\bit
	
	\item [(i)] The LICQC holds for \reff{1.1} at $x^*$
	if and only if it holds for \reff{3.5} at $\tilde{x}^*$.
	
	\item [(ii)] Suppose the LICQC holds for \reff{1.1} at $x^*$.
	Then, the SCC holds for \reff{1.1} at $x^*$
	if and only if it holds for \reff{3.5} at $\tilde{x}^*$.

	\item [(iii)] Suppose the LICQC holds for \reff{1.1} at $x^*$.
	Then, the SOSC holds for \reff{1.1} at $x^*$
	if and only if it holds for \reff{3.5} at $\tilde{x}^*$.
	\eit
\end{theorem}
\begin{proof}
	Without loss of generality, one can assume $f_{\min}=0$
	up to shifting a constant in $f$. By Lemma~\ref{lemma3.1},
	when $K$ is closed at $\infty$, $f(x) \ge 0$ on $K$
	if and only if $\tilde{f}(\tilde{x}) \ge 0$ on $\widetilde{K}$,
	which is the feasible set of \reff{3.5}.
	Note that $x^*$ is a  minimizer of \reff{1.1} if and only if $f(x^*)=0$.
	Since $\tilde{f}(\tilde{x}^*)  = (1+\|x^*\|^2)^{-d/2} f(x^*)$,
	$\tilde{f}(\tilde{x}^*) = 0$ if and only if $f(x^*) = 0$.
	Hence, $x^*$ is a  minimizer of \reff{1.1}
	if and only if $\tilde{x}^*$ is a  minimizer of \reff{3.5}.
	Since the constraint $x_0 \ge 0$ is not active at $\tilde{x}^*$,
	the label set of active constraints is
	\[
	J(x^*) \coloneqq  \{i \in \mathcal{E} \cup \mathcal{I}:  c_{i}(x^*)=0 \}.
	\]
	By the Euler's identity for homogeneous polynomials,
	we have that for all $\tilde{x}$
	\[
	x^\mathrm{T} \nabla_x \tilde{c}_i(\tilde{x}) + x_0 \nabla_{x_0} \tilde{c}_i(\tilde{x})
	= \deg(c_i) \cdot  \tilde{c}_i(\tilde{x}).
	\]
	The above implies that
	\be \label{Euler:x*:x0=1}
	(x^*)^\mathrm{T} \nabla_{x} \tilde{c}_i(\tilde{x}^*) +
	\nabla_{x_0} \tilde{c}_i(\tilde{x}^*) = 0
	\ee
	for all $i \in J(x^*)$. Also note that
	\be \label{gx(tldc):gx(c):x*}
	\nabla_x \tilde{c}_i(\tilde{x}^*) =
	(\sqrt{1+\|x^*\|^2})^{1-\deg(c_i)} \nabla c_i(x^*) .
	\ee

	\smallskip
	\noindent
	(i)``$\Leftarrow$":
	Suppose the LICQC holds for \reff{3.5} at $\tilde{x}^*$, i.e., the gradients
	\[
	\bbm
	\nabla_{x_0} \tilde{c}_i (\tilde{x}^*) \\
	\nabla_{x}\tilde{c}_i(\tilde{x}^*)
	\ebm (i \in J(x^*)), \quad
	\bbm 1 \\ x^*  \ebm
	\]
	are linearly independent. Then \reff{Euler:x*:x0=1} implies that
	$\nabla_{x}\tilde{c}_i(\tilde{x}^*)\, (i \in J(x^*) )$
	are linearly independent.
	By \reff{gx(tldc):gx(c):x*}, we know the gradients
	$\nabla c_{i}(x^*)$ $(i \in J(x^*))$ are linearly independent,
	i.e., the LICQC holds for \reff{1.1} at $x^*$.
	
	\smallskip \noindent
	``$\Rightarrow$":
	Assume the LICQC holds for \reff{1.1} at $x^*$, i.e., the gradients
	$\nabla c_{i}(x^*)$ $(i \in J(x^*))$ are linearly independent.
	Suppose there are scalars $\mu_i$, $\mu_{0}$ such that
	\be \label{lipend:mui}
	\sum_{i \in J(x^*)}\mu_i
	\bbm
	\nabla_{x_0}\tilde{c}_{i}(\tilde{x}^*) \\
	\nabla_{x}\tilde{c}_{i}(\tilde{x}^*)
	\ebm
	+ \mu_0 \bbm 1 \\ x^*  \ebm  =0.
	\ee
	Note that $\tilde{c}_i(\tilde{x}^*)= 0$ for all $i \in J(x^*)$.
	Premultiplying $(\tilde{x}^*)^\mathrm{T}$ in \reff{lipend:mui} results in
	\[
	0 =\sum_{i \in J(x^*)}\mu_i \deg(c_i) \cdot \tilde{c}_i(\tilde{x}^*)
	+ \mu_0 \sqrt{1+\|x^*\|^2} =\mu_0 \sqrt{1+\|x^*\|^2} .
	\]
	Hence, we get $\mu_0=0$.
	By \reff{gx(tldc):gx(c):x*}, the linear independence of
	$\nabla c_{i}(x^*)$ $(i \in J(x^*))$
	implies that $\mu_i=0$ for all $i \in J(x^*)$.
	So, the LICQC holds for \reff{3.5} at $\tilde{x}^*$.
	
	\medskip
	\noindent (ii)
	By the item (i), the LICQC holds at $\tilde{x}^*$ for \reff{3.5}.
	Since $\tilde{x}^*$ is a minimizer of \reff{3.5},
	there exist Lagrange multipliers $\lambda_i$ $(i \in J(x^*))$
	and $\lambda_{0}$ such that
	\be \label{FOOC:(4.1)}
	\bbm
	\nabla_{x_0}  F(\tilde{x}^*) \\
	\nabla_{x}  F(\tilde{x}^*)
	\ebm   = \sum_{i \in J(x^*)}\lambda_i
	\bbm
	\nabla_{x_0}  \tilde{c}_i (\tilde{x}^*) \\
	\nabla_{x} \tilde{c}_i(\tilde{x}^*)
	\ebm + \lambda_0  \bbm 1 \\ x^*  \ebm .
	\ee
	By the Euler's identity for homogeneous polynomials, we have
	\[
	(\tilde{x}^*)^\mathrm{T} \nabla_{ \tilde{x} } F(\tilde{x}^*) =0, \quad
	(\tilde{x}^*)^\mathrm{T}  \nabla_{ \tilde{x} } \tilde{c}_i(\tilde{x}^*)
	=0~(i \in J(x^*)).
	\]
	The above and \reff{FOOC:(4.1)} imply that $\lmd_0=0$. Let
	$\hat{\lambda}_{i}=\lambda_i (\sqrt{1+\|x^*\|^2})^{d-\deg(c_i)}$, then
	\be \label{g(fci):hatlmd:x*}
	\nabla f(x^*) = \sum_{ i \in J(x^*) }
	\hat{\lambda}_{i}\nabla c_i(x^*).
	\ee
	When the LICQC holds, the Lagrange multipliers are unique
	for both \reff{1.1} and \reff{3.5}.
	
	\smallskip
	``$\Rightarrow$": Suppose the SCC holds for \reff{1.1} at $x^*$.
	If otherwise the SCC fails to hold for \reff{3.5} at $\tilde{x}^*$,
	say, $\lambda_{j_0}  =0$ for some $j_0 \in J(x^*) \cap \mathcal{I}$,
	then as in the above we can get
	\[
	\nabla f(x^*) = \sum_{i \in J(x^*)\backslash \{j_0\}}
	\hat{\lambda}_i \nabla c_i(x^*) .
	\]
	Since the Lagrange multipliers are unique, the above implies that
	the SCC fails for \reff{1.1} at $x^*$, which is a contradiction.
	Therefore, the SCC holds for \reff{3.5} at $\tilde{x}^*$.
	
	\smallskip
	``$\Leftarrow$": Suppose the SCC holds for \reff{3.5} at $\tilde{x}^*$,
	then $\lambda_i>0$ for all $i \in J(x^*)\cap \mathcal{I}$ in \reff{FOOC:(4.1)}.
	So, $\hat{\lambda}_i>0$ for each
	$i \in J(x^*)\cap \mathcal{I}$ in \reff{g(fci):hatlmd:x*}.
	Since the Lagrange multipliers are unique,
	this means that the SCC holds for \reff{1.1} at $x^*$.

	\medskip
	\noindent (iii)	
	Note that the LICQC holds for both \reff{1.1} and \reff{3.5}.

	\smallskip \noindent
	``$\Rightarrow$":
	Suppose the SOSC holds \reff{1.1} at $x^*$.
	Consider $\tilde{y} \coloneqq (y_0,y)$ in the tangent space
	of active constraints of \reff{3.5} at $\tilde{x}^*$, i.e.,
	\be \label{tangent:y:hmg}
	\boxed{
		\baray{c}
		y_0 \nabla_{x_0} \tilde{c}_i (\tilde{x}^*) +
		y^\mathrm{T}\nabla_{x}\tilde{c}_i(\tilde{x}^*) = 0 ~~ (i \in J(x^*)), \\
		y_0 + y^\mathrm{T}x^* =  0
		\earay }.
	\ee
	The equation \reff{Euler:x*:x0=1} implies that
	$\nabla_{x_0} \tilde{c}_i(\tilde{x}^*) =
	-(x^*)^\mathrm{T}\nabla_x \tilde{c}_i(\tilde{x}^*)$
	for every $i \in J(x^*)$.
	So \reff{tangent:y:hmg} is equivalent to
	\be  \label{4C}
	\boxed{
		\baray{c}
		(y-y_0x^*)^\mathrm{T}\nabla_{x}\tilde{c}_i(\tilde{x}^*)
		\, = 0 \, (i \in J(x^*) ), \\
		y_0 + y^\mathrm{T}x^* \, = \, 0
		\earay }.
	\ee
	Let $\lambda_{i}, \hat{\lambda}_i$ be Lagrange multipliers as in
	\reff{FOOC:(4.1)}-\reff{g(fci):hatlmd:x*} for the proof of item (ii).
	Note that $\lambda_{0}=0$ and we shift $f$ as $f_{\min}=0$, so we have
	
	\be  \label{4D}
	(y-y_0x^*)^\mathrm{T}\nabla_{x} \tilde{f}(\tilde{x}^*) =
	\sum_{i \in J(x^*)}\lambda_i(y-y_0x^*)^\mathrm{T}
	\nabla_{x}\tilde{c}_i(\tilde{x}^*)= 0 .
	\ee
	For the Lagrange function
	\be \label{Lagfun:lmd}
	L(\tilde{x})  \coloneqq  \tilde{f}(\tilde{x}) -
	\sum_{i \in J(x^*)} \lambda_i \tilde{c}_i(\tilde{x}),
	\ee
	its Hessian at $\tilde{x}^*$ is
	\begin{equation*}
		\nabla_{\tilde{x}\tilde{x}}^2L(\tilde{x}^*)=\left[
		\begin{array}{cc}
			\nabla^2_{x_0,x_0}L(\tilde{x}^*)   & \nabla^2_{x,x_0}L(\tilde{x}^*)  \\
			\nabla^2_{x,x_0}L(\tilde{x}^*)^{\mathrm{T}} & \nabla_{x,x}^2L(\tilde{x}^*) \\
		\end{array}\right].
	\end{equation*}
	We express $\nabla_{\tilde{x}\tilde{x}}^2L(\tilde{x}^*)$
	in terms of $\frac{\partial^2 L(\tilde{x}^*)}{\partial x_i x_j}$
	($1\leq i \leq n$, $1\leq j\leq n$). Let
	\[
	\eta_0  \coloneqq \frac{1}{\sqrt{1+\|x^*\|^2}}, \quad
	\eta  \coloneqq  \frac{x^*}{\sqrt{1+\|x^*\|^2}}.
	\]
	Similar to \reff{Euler:x*:x0=1}, it holds that
	\be \nonumber
	\boxed{
		\baray{c}
		\eta_0\nabla_{x_0} \tilde{f}(\tilde{x}^*) +
		\eta^\mathrm{T}\nabla_x \tilde{f}(\tilde{x}^*) = 0, \\
		\eta_0\nabla_{x_0} \tilde{c}_i(\tilde{x}^*) +
		\eta^\mathrm{T}\nabla_x \tilde{c}_i(\tilde{x}^*) = 0
		\, (i \in J(x^*) )
		\earay
	}.
	\ee
	Again, by Euler's identity, we similarly have
	\be \label{Euler:eta0:eta}
	(d-1)\nabla_{x_{j}}\tilde{f}(\tilde{x}^*)=
	\bbm
	\nabla^2_{x_{j},x_0}\tilde{f}(\tilde{x}^*) &
	\nabla^2_{x_{j},x}\tilde{f}(\tilde{x}^*)
	\ebm
	\bbm \eta_0 \\ \eta \ebm ,
	\ee
	\[
	(\deg(c_i)-1)\nabla_{x_{j}} \tilde{c}_i(\tilde{x}^*)=
	\bbm
	\nabla^2_{x_{j},x_0} \tilde{c}_i(\tilde{x}^*) &
	\nabla^2_{x_{j},x}\tilde{c}_i(\tilde{x}^*)
	\ebm
	\bbm \eta_0 \\ \eta \ebm,
	\]	
	for all $i \in J(x^*)$ and $j =0,\dots,n$.
	For convenience, denote
	\[
	v^*  \coloneqq  d\nabla_{x}\tilde{f}(\tilde{x}^*) -
	\sum_{i \in J(x^*)}\deg(c_i)\lambda_{i}\nabla_{x}\tilde{c}_i(\tilde{x}^*),
	\]
	\[
	\gamma \coloneqq  d\nabla_{x_0}\tilde{f}(\tilde{x}^*) -
	\sum_{i \in J(x^*)}\deg(c_i)\lambda_{i}\nabla_{x_0}\tilde{c}_i(\tilde{x}^*).
	\]
	With \reff{FOOC:(4.1)} and \reff{Euler:eta0:eta}, we have
	\be
	\begin{split}
		\eta_0\nabla^2_{x,x_0} L(\tilde{x}^*)&=(d-1)\nabla_{x}\tilde{f}(\tilde{x}^*)-\sum_{i \in J(x^*)}(\deg(c_i)-1)\lambda_{i}\nabla_{x}\tilde{c}_i(\tilde{x}^*)-\nabla^2_{x,x}L(\tilde{x}^*)\eta\\
		&=v^*- \nabla^2_{x,x}L(\tilde{x}^*)\eta,
		\label{4A}
	\end{split}
	\ee
	
	\be
	\eta_0\gamma=-d\nabla_{x}\tilde{f}(\tilde{x}^*)^\mathrm{T}\eta +
	\sum_{i \in J(x^*)}\deg(c_i)\lambda_{i}\nabla_{x}\tilde{c}_i(\tilde{x}^*)^\mathrm{T}\eta=-\eta^{\mathrm{T}}v^*.
	\ee
	Hence, the following holds
	\be
	(\eta_0)^2\nabla^2_{x_0,x_0}L(\tilde{x}^*) =
	\eta_0\gamma-\eta_0\nabla^2_{x_0,x} L(\tilde{x}^*)\eta =
	-2\eta^{\mathrm{T}}v^*+\eta^{\mathrm{T}}\nabla^2_{x,x}L(\tilde{x}^*)\eta .
	\label{4B}
	\ee
	Consider the new Lagrange function
	\be \label{Lagfun:hat:lmd}
	\hat{L}(x)  \coloneqq  f(x) -
	\sum_{i \in J(x^*)} \hat{\lambda}_i c_i(x) .
	\ee
	Since the SOSC holds for \reff{1.1} at $x^*$, we have
	\begin{equation*}
		y^{\mathrm{T}}\nabla^2\hat{L}(x^*)y >  0, \quad
		\mbox{for each} \,\, y \ne 0, \,
		y^{\mathrm{T}} \nabla c_i(x^*)=0\, ~(i \in J(x^*)).
	\end{equation*}
	The above is equivalent to
	\begin{equation*}
		y^{\mathrm{T}}\nabla_{xx}^2L(\tilde{x}^*)y >  0, \quad
		\mbox{for each} \,\, y \ne 0, \,
		y^{\mathrm{T}} \nabla_x \tilde{c}_i(\tilde{x}^*)=0\, ~(i \in J(x^*)).
	\end{equation*}
	Using \reff{4A}-\reff{4B}, one can verify that for
	$\tilde{y} = (y_0, y)$ satisfying \reff{4C},
	\begin{equation}
		\label{sosc}
		\begin{split}
			\tilde{y}^\mathrm{T}\nabla_{\tilde{x}\tilde{x}}^2L(\tilde{x}^*)\tilde{y}&=  y^\mathrm{T}\nabla_{xx}^2L(\tilde{x}^*)y+2y_0y^\mathrm{T}\nabla^2_{x,x_0} L(\tilde{x}^*)+y_0^2\nabla^2_{x_0,x_0}L(\tilde{x}^*)\\
			&=(y-y_0x^*)^\mathrm{T}\nabla_{x,x}^2L(\tilde{x}^*)(y-y_0x^*)+
			\frac{2y_0(y-y_0x^*)^\mathrm{T}v^*}{\eta_0}\\
			&=(y-y_0x^*)^\mathrm{T}\nabla_{x,x}^2L(\tilde{x}^*)(y-y_0x^*)\geq 0 .
		\end{split}
	\end{equation}
	In the third equality above, $(y-y_0x^*)^\mathrm{T}v^*=0$
	follows from  $(\ref{4C})$-$(\ref{4D})$.
	Moreover,  if
	$
	\tilde{y}^\mathrm{T}\nabla_{\tilde{x}\tilde{x}}^2L(\tilde{x}^*)\tilde{y}=0,
	$
	then $ y-y_0x^*=0$, by the SOSC for \reff{1.1} at $x^*$.
	Since $y_0 + y^\mathrm{T}x^* = 0$ as in \reff{4C}, we get
	\[
	0 = (y-y_0x^*)^\mathrm{T} x^* = -y_0 (1+\|x^*\|^2).
	\]
	So, $y_0 = 0$ and hence $y = 0$, i.e., $\tilde{y}=0$.
	This shows that the SOSC holds for \reff{3.5} at $\tilde{x}^*$.

	\smallskip \noindent
	``$\Leftarrow$":
	When the SOSC holds for \reff{3.5} at $\tilde{x}^*$,
	we show that it also holds for \reff{1.1} at $x^*$.
	Suppose otherwise the SOSC fails to hold for \reff{1.1} at $x^*$,
	then there exists  $u \neq 0 $ such that
	\[
	u^{\mathrm{T}}\nabla^2\hat{L}(x^*) u=0,~
	u^{\mathrm{T}} \nabla c_i(x^*)=0 ~(i \in J(x^*)).
	\]
	Let $\tilde{\ell} \coloneqq (\ell_0,\ell)$, where
	$\ell_0 = -u^\mathrm{T}x^* / (1+\|x^*\|^2)$ and $\ell = u+\ell_0x^*.$
	For the case $\ell_0= 0$, $\ell=u$ is nonzero.
	So, $\tilde{\ell} $ is a nonzero vector. Note that
	\begin{equation*}
		\baray{c}
		(\ell-\ell_0x^*)^\mathrm{T}\nabla_{x}\tilde{c}_i(\tilde{x}^*)=
		u^\mathrm{T}\nabla_{x}\tilde{c}_i(\tilde{x}^*)=0 ~(i \in J(x^*)), \\
		\ell^\mathrm{T}x^*+\ell_0=u^\mathrm{T}x^*+\ell_0(1+\|x^*\|^2)=0.
		\earay
	\end{equation*}
	By equation \reff{4C}, we know  $\tilde{\ell}$ lies in the tangent space
	of active constraints of \reff{3.5} at $\tilde{x}^*$.
	Moreover, similar as in \reff{sosc},  the following holds
	\begin{equation*}
		\begin{split}		
			\tilde{\ell}^\mathrm{T}\nabla_{\tilde{x}\tilde{x}}^2L(\tilde{x}^*)\tilde{\ell}&
			=(\ell-\ell_0x^*)^\mathrm{T}\nabla_{xx}^2L(\tilde{x}^*)(\ell-\ell_0x^*)\\
			&= u^{\mathrm{T}}\nabla_{\tilde{x}\tilde{x}}^2L(\tilde{x}^*)u=
			u^{\mathrm{T}}\nabla^2\hat{L}(x^*)u = 0.
		\end{split}	
	\end{equation*}
	This is a contradiction to that
	the SOSC holds at $\tilde{x}^*$ for \reff{3.5}.
	So the SOSC must hold for \reff{1.1} at $x^*$.
\end{proof}

The LICQC, SCC and SOSC all hold at every local minimizer of \reff{1.1}
for generic polynomial optimization problems.
This is a major conclusion in \cite{nieopcd}.
By Theorem~\ref{lemma3.11},
we know these optimality conditions all hold at every
regular minimizer of \reff{3.5}
for generic polynomials.

\subsection{Optimality conditions for minimizers at infinity}
\label{ssc:infinity}

We consider minimizers at infinity.
For a polynomial $p$, recall that $p^\hm$
denotes the homogeneous part of the highest degree for $p$.
Suppose $K$ is closed at $\infty$ and $\tilde{x}^*=(0,x^*)$
is a minimizer at $\infty$ for \reff{3.5}.
By Theorem~\ref{thm:prop:inf} and Lemma~\ref{hom:opt:zero},
we know $f^\hm(x^*)=0$ and $x^*$ is a minimizer for
\be   \label{p3.7}
\left\{ \begin{array}{rl}
	\min   & f^\hm(x) \\
	\st & c_{i}^\hm (x) = 0 ~(i \in \mathcal{E}),\\
	& c_{j}^\hm (x) \geq  0 ~(j \in \mathcal{I}), \\
	&  \| x \|^2 - 1 = 0 .
\end{array} \right.
\ee
Denote the active label set
\be    	\label{jx}
J_1(x^*)   \coloneqq
\left\{i \in \mathcal{E} \cup \mathcal{I} \mid
c^\hm_{i}(x^*)=0\right\}.
\ee
By the Fritz-John  condition (see \cite{Bert97}),
there exist scalars $\mu_0$, $\mu_i$, $\bar{\mu}$ such that
\[
\mu_0 \nabla f^\hm(x^*)=\sum_{i \in J_1(x^*)}
\mu_i \nabla c^\hm_i(x^*) + \bar{\mu} x^*,
\]
while $\mu_0$, $\mu_i$, $\bar{\mu}$ are not all zero.
Note that $\| x^* \| = 1$.
Since  $f^\hm(x^*)=0$ and $c_i^\hm (x^*)=0$
for each $i \in J_1(x^*)$, premultiplying $(x^*)^\mathrm{T}$
in the above results in $\bar{\mu}=0$,
by the Euler's identity.
So the above is the same as
\be  \label{eq3.11}
\mu_0 \nabla f^\hm(x^*)=\sum_{i \in J(x^*)}
\mu_i \nabla c^\hm_i(x^*) .
\ee

First, we show that
there are no minimizers at infinity  for \reff{1.1}
if the polynomials are generic.
For this, we review some basic theory for resultants and discriminants.
We refer to \cite{GKZ94,njwdis,Stu02}.
Let $p_1, \dots, p_n$ be forms in $x \coloneqq (x_1,\dots,x_n)$.
The resultant $\operatorname{Res}(p_1, \ldots, p_n)$
is a  polynomial in the coefficients of $p_1, \dots, p_n$
such that
\begin{equation*}
	\operatorname{Res}(p_1, \ldots, p_n) =0 \Leftrightarrow
	\exists~ 0 \neq u \in \mathbb{C}^{n},~ p_{1}(u)=\cdots=p_{n}(u)=0.
\end{equation*}
For $m \leq n$, the discriminant $\Delta(p_1, \ldots, p_m)$
is a polynomial in the coefficients of $p_1, \dots, p_m$
such that $\Delta(p_1, \ldots, p_m)=0$ if and only if
there exists $0 \neq u \in \mathbb{C}^{n}$ satisfying
\begin{equation*}
	p_{1}(u)=\cdots=p_{m}(u)=0, ~
	\operatorname{rank}\left[\nabla p_{1}(u) ~ \cdots ~ \nabla p_{m}(u)\right]<m.
\end{equation*}
Both $\operatorname{Res}(p_1, \ldots, p_n)$ and
$\Delta(p_1, \ldots, p_m)$
are homogeneous polynomials in the coefficients of $p_i$.
The following is about nonexistence of
minimizers at infinity for generic polynomial optimization problems.

\begin{theorem}  \label{thm4.5}
	Suppose $K$ is closed at $\infty$ and  $\mathcal{E}=\{1,\dots,m_1\}$,  $\mathcal{I}=\{m_1+1,\dots,m_2\}$, $m_1\leq n-1$. If the polynomials
	$f \in \mathbb{R}[x]_{d_0}$ and $c_{i} \in
	\mathbb{R}[x]_{d_{i}}\left(i \in \mathcal{E} \cup \mathcal{I}\right)$
	satisfy the conditions:
	
	\bit
	
	\item [(i)] For all $m_1+1 \leq j_{1}<\cdots<j_{n-m_1} \leq m_2$
	\[
	\operatorname{Res}(c_1^\hm,\dots,c_{m_1}^\hm,
	c^\hm_{j_{1}}, \ldots, c^\hm_{j_{n-m_1}}) \neq 0 ;
	\]
	
	\item [(ii)] For all $m_1+1 \leq j_{1}<\cdots<j_{r} \leq m_2$
	with $0 \leq r \leq n-m_1-1$
	\[
	\Delta(f^\hm, c_1^\hm,\dots,c_{m_1}^\hm, c^\hm_{j_{1}}, \ldots,
	c^\hm_{j_{r}}) \neq 0,
	\]
	\eit 	
	then \reff{1.1} has no minimizers at $\infty$.
\end{theorem}
\begin{proof}
	Since $K$ is closed at $\infty$,
	the optimal value of \reff{p3.7} is $0$
	if there is a minimizer  at infinity $x^*$.
	Let $j_1,\dots,j_r$ be the labels of active inequality constraints
	for $(\ref{p3.7})$ at $x^*$.
	The item $(a)$ implies that $r < n- m_1$.
	This is because if otherwise $r \geq n-m_1$, then
	\[
	\operatorname{Res}( c_1^\hm,\dots,c_{m_1}^\hm, c^\hm_{j_{1}}, \ldots,
	c^\hm_{j_{n-m_1}} )  =  0.
	\]
	The Fritz-John condition \reff{eq3.11} implies that
	\[
	\Delta(f^\hm, c_1^{(1)},\dots,c_{m_1}^\hm, c^\hm_{j_{1}}, \ldots,
	c^\hm_{j_{r}}) = 0 .
	\]
	The conditions in (i)-(ii) deny existence of minimizers at infinity.
\end{proof}

For special polynomial optimization problems,
there may exist minimizers at infinity.
For instance, this is the case if
the optimal value is not achievable for \reff{1.1}
(see Theorem~\ref{thm:prop:inf}).
In the following,  we study optimality conditions
for minimizers at infinity.
Suppose $\tilde{x}^* \coloneqq (0, x^*)$
is a minimizer at infinity for \reff{3.5}. Recall that $d = \deg(f)$.
The KKT equation for \reff{3.5} at $\tilde{x}^*$ is in the form
\be \label{KKT:inf:(0,x*)}
\baray{rcr}
\left[\begin{array}{c}
	f^\se(x^*)-d \cdot f_{\min}\cdot 0^{d-1}  \\ 	\nabla f^\hm(x^*)
\end{array} \right]
& = & \sum\limits_{i \in J_1(x^*)} \lambda_i
\left[\begin{array}{c}
	c_{i}^\se\left(x^*\right) \\	\nabla c_{i}^\hm(x^*)
\end{array}\right] + \qquad  \\
& &  \lambda_0 \bbm  1 \\ 0 \ebm +
\bar{\lambda} \bbm 	0 \\ x^* \ebm.
\earay
\ee

\begin{theorem}  \label{p4.7}
	Suppose $K$ is closed at $\infty$ and $\tilde{x}^* \coloneqq (0, x^*)$
	is a minimizer at infinity for \reff{3.5}. Let
	$J_1(x^*)$ be as in $(\ref{jx})$. Then, we have:
	
	\bit
	
	\item [(i)] The LICQC holds for \reff{3.5} at $\tilde{x}^*$
	if and only if the gradients
	\[  \nabla c_{i}^\hm(x^*)~({i \in J_1(x^*)}) \]
	are linearly independent.

	\item [(ii)] Suppose the LICQC holds  for \reff{3.5} at $\tilde{x}^*$.
	Then, in \reff{KKT:inf:(0,x*)}, we have
	\be \label{opcdinf:lmdbar0}
	\lambda_{0}=f^\se(x^*)-d \cdot f_{\min}\cdot 0^{d-1} -
	\sum_{i \in J_1(x^*)} \lambda_ic_{i}^\se(x^*),
	\quad  \bar{\lambda} = 0.
	\ee
	Moreover, the SCC holds  for \reff{3.5} at $\tilde{x}^*$  if and only if
	\be \label{SCC:inf:x*}
	\lambda_0>0, ~ \lambda_i>0~(i \in  J_1(x^*)\cap \mathcal{I}).
	\ee
	
	\item [(iii)] Suppose the LICQC holds for \reff{3.5} at $\tilde{x}^*$. Let $\lmd_i$ be Lagrange multipliers as in \reff{KKT:inf:(0,x*)}.
	Then, the SOSC holds for \reff{3.5} at $\tilde{x}^*$
	if and only if for every nonzero $y $ satisfying
	\be \label{hom:inf:s}
	\boxed{
		\baray{c}
		y^{\mathrm{T}}\nabla c_{i}^\hm(x^*)=0 ~(i \in J_1(x^*)), \\
		y^{\mathrm{T}}x^*=0
		\earay
	},
	\ee
	we have $ y^\mathrm{T}\nabla^2L_1(x^*)y > 0$, where
	\[
	L_1(x)= f^\hm(x)-\sum_{i \in J_1(x^*)} \lambda_i c_i^\hm (x).
	\]
	
	\eit
\end{theorem}
\begin{proof}
	\par (i)
	The constraint $x_0\geq 0$ is active \reff{3.5} at $\tilde{x}^*$.
	The LICQC at $\tilde{x}^*$ for \reff{3.5}
	requires the linear independence of the gradients
	\begin{equation*}
		\bbm
		c_{i}^\se\left(x^*\right) \\ \nabla c_{i}^\hm(x^*)
		\ebm~({i \in J_1(x^*)}) ,  \,\,
		\bbm  1 \\ 0 \ebm, \,\,
		\bbm  0 \\ x^* \ebm .
	\end{equation*}
	
	\smallskip \noindent
	``$\Rightarrow$":
	Suppose the LICQC holds at $\tilde{x}^*$ for \reff{3.5}.
	If there is a linear combination such that
	$
	\sum_{i \in J_1(x^*)} \mu_i \nabla c_{i}^\hm(x^*)=0,
	$
	then
	\begin{equation*}
		\sum_{i \in J_1(x^*)} \mu_i
		\bbm
		c_{i}^\se\left(x^*\right)  \\    \nabla c_{i}^\hm(x^*)
		\ebm +
		\mu_0 \bbm 1 \\ 0 \ebm  =  0
	\end{equation*}
	for $\mu_0 = -\sum_{i \in J_1(x^*)} \mu_i c_{i}^\se(x^*)$.
	The LICQC holds at $\tilde{x}^*$ for \reff{3.5}
	implies that all $\mu_i=0$. Hence,
	the gradients $\nabla c_{i}^\hm(x^*)~({i \in J_1(x^*)})$
	are linearly independent.

	\smallskip \noindent
	``$\Leftarrow$":
	Suppose the gradients $\nabla c_{i}^\hm(x^*)~({i \in J_1(x^*)})$
	are linearly independent.
	Consider a linear combination such that
	\be  \label{o1}
	\sum_{i \in J_1(x^*)} \mu_i
	\bbm
	c_{i}^\se\left(x^*\right) \\ \nabla c_{i}^\hm(x^*)
	\ebm + \mu_0 \bbm 1 \\ 0 \ebm +
	\bar{\mu} \bbm 0 \\ x^* \ebm = 0.
	\ee
	Since $c_i^\hm(x^*)= 0$ for each $i \in J_1(x^*)$, we have
	$(x^*)^\mathrm{T}\nabla c_{i}^\hm(x^*)=0$.
	Premultiplying $(\tilde{x}^*)^\mathrm{T}$ in the above results in
	$\bar{\mu} \|x^*\|^2=0$, so $\bar{\mu}=0$.
	The linear independence of $\nabla c_{i}^\hm(x^*)~({i \in J(x^*)})$
	implies that $\mu_i=0$ for all $i \in J_1(x^*) $.
	Finally, we get $\mu_0=0$. So, the LICQC holds at $\tilde{x}^*$ for \reff{3.5}.

	\smallskip \noindent
	(ii) Note that $ \tilde{x}^*=(0,x^*)$ is a minimizer of \reff{3.5}.
	Since the LICQC holds at $\tilde{x}^*$,
	there exist scalars $\bar{\lambda}, \lambda_0$,
	$\lambda_i$ $(i \in J_1(x^*) )$ satisfying \reff{KKT:inf:(0,x*)}.
	Premultiplying $(\tilde{x}^*)^\mathrm{T}$ in \reff{KKT:inf:(0,x*)},
	by the Euler's identity, we can express
	$\lmd_0$, $\bar{\lambda}=0$ as in \reff{opcdinf:lmdbar0}.
	The SCC for \reff{3.5} at $\tilde{x}^*$ is equivalent to \reff{SCC:inf:x*}.

	(iii) Since the LICQC holds, there exist Lagrange multipliers
	$ \lambda_0$, $\lambda_i$ as in \reff{KKT:inf:(0,x*)}.
	The Lagrange function for \reff{3.5} is
	\[
	L(\tilde{x})  \coloneqq
	F(\tilde{x})-  \sum_{i \in J_1(x^*)}
	\lambda_i \tilde{c}_i(x)-\lambda_0x_0.
	\]
	Its Hessian expression at $\tilde{x}^*$ is
	\begin{equation*}
		\nabla^2_{\tilde{x},\tilde{x}} L(\tilde{x}^*)=
		\bbm
		L_3(x^*)  & \nabla L_2(x^*)^{\mathrm{T}} \\
		\nabla L_2(x^*)& \nabla^2 L_1(x^*)\\
		\ebm,
	\end{equation*}
	where
	\[
	L_2(x)=f^\se(x)-\sum_{i \in J_1(x^*)} \lambda_i c_{i}^\se(x),
	\]
	\[L_3(x) =2f^{\thi}(x)-d(d-1) \cdot f_{\min}\cdot 0^{d-2} -
	\sum_{i \in J_1(x^*)} 2\lambda_i c_{i}^\thi(x).
	\]
	Consider $\tilde{y} \coloneqq (y_0,y)$ in the tangent space
	of active constraints of \reff{3.5} at $\tilde{x}^*$, i.e.,
	\be \label{hom:inf}
	\boxed{ \baray{c}
		y_0c_{i}^\se\left(x^*\right)+y^{\mathrm{T}}
		\nabla c_{i}^\hm(x^*)=0 ~(i \in J(x^*)),\\
		y^{\mathrm{T}}x^*=0, ~ y_0=0
		\earay }.
	\ee
	Note that \reff{hom:inf} is equivalent to \reff{hom:inf:s}.
	The SOSC for \reff{3.5} at $\tilde{x}^*$ requires:
	for every $\tilde{y} \ne 0$ satisfying \reff{hom:inf},
	it holds that
	\[
	\tilde{y}^\mathrm{T}\nabla^2_{\tilde{x}\tilde{x}}
	L(\tilde{y}^*)\tilde{y} \, = \,
	y^\mathrm{T}\nabla^2L_1(x^*)y \, > \, 0.
	\]
	The above is equivalent to that $y^\mathrm{T}\nabla^2L_1(x^*)y > 0$
	for every $y \ne 0$ satisfying \reff{hom:inf:s}.
	So the conclusion of item (iii) holds.
\end{proof}

\subsection{The even degree case}
\label{ssc:evendeg}

When $f$ and $c_j$ $(j\in \mathcal{I})$ all have even degrees,
the homogenization \reff{3.2} is equivalent to
\reff{lo3.5}; see Proposition \ref{prop3.4}.
This means that the constraint $x_0 \ge 0$
is redundant for \reff{3.5}. Typically, the SCC fails to hold
for minimizers at infinity.
For this case, Theorem~\ref{thm3.12} is not applicable
for showing the finite convergence
for the hierarchy of \reff{3.3}-\reff{d3.3}.
However, the same conclusion like in Theorem~\ref{thm3.12}
holds under optimality conditions without the constraint $x_0 \ge 0$.

The $k$th order SOS relaxation for \reff{lo3.5} is
\be  \label{hmg:even:SOS}
\left\{ \baray{rl}
\max &   \gamma \\
\st &  \tilde{f}(\tilde{x})- \gamma x_0^d \in
\ideal{\tilde{c}_{\mathcal{E}}}_{2k} +\qmod{\tilde{c}_{\mathcal{I}^{*}}}_{2k}.
\earay \right.
\ee
Its dual optimization is the $k$th order moment relaxation
\be  \label{hmg:even:MOM}
\left\{ \baray{rl}
\min  &  \langle \tilde{f}, y \rangle  \\
\st  & L_{p}^{(k)}[y] =0\, ~(p \in \tilde{c}_{\mathcal{E}}),\\
& L_{q}^{(k)}[y] \succeq 0\, ~(q \in \tilde{c}_{\mathcal{I}^{*}}), \\
&  M_k[y] \succeq 0, \\
& \langle x_0^d, y \rangle = 1, \,
y \in \re^{ \N^{n+1}_{2k} }.
\earay \right.
\ee
In the above, the polynomial tuple $\tilde{c}_{\mathcal{I}^{*}}$ is
\be \label{sets:tld:cie}
\tilde{c}_{\mathcal{I}^{*}}  \coloneqq  \big\{ \tilde{c_j}(\tilde{x}) \big\}_{j \in \mc{I} } .
\ee
Let $f_k^e$ and $f^{e,\prime}_k$ denote the optimal values of
\reff{hmg:even:SOS}, \reff{hmg:even:MOM} respectively.
To study their convergence,
we consider the optimization problem

\begin{equation}  \label{hom:dis}
	\left\{ \baray{rl}
	\min  & F(\tilde{x}) =  \tilde{f}(\tilde{x}) - f_{\min} \cdot x_0^d  \\
	\st &   \tilde{c}_{i}(\tilde{x})=0~(i \in \mathcal{E}),\\
	&  \tilde{c}_{j}(\tilde{x}) \geq 0~(j \in \mathcal{I}),\\
	&  x_0^2 + \|x\|^2-1=0.
	\earay \right.
\end{equation}
For the even degree case, the LICQC, SCC and SOSC
are said to hold at a minimizer at infinity $x^*$  if they hold for
\reff{hom:dis} at $(0,x^*)$.
The following is the convergence for the hierarchy of
relaxations \reff{hmg:even:SOS}-\reff{hmg:even:MOM}.

\begin{theorem}  \label{thm:ficvg:even}
	Assume $K$ is closed at $\infty$, $\ideal{\tilde{c}_{\mathcal{E}}}$ is real radical,
	$f$ and $c_j$ $(j\in \mathcal{I})$ all have even degrees.
	If the LICQC, SCC and SOSC hold at
	every minimizer of \reff{1.1}, including the one at infinity,
	then the hierarchy of relaxations
	\reff{hmg:even:SOS}-\reff{hmg:even:MOM} is tight,
	i.e., there exists $k_0$ such that
	\[
	f_k^e  =  f^{e,\prime}_k  =  f_{\min}
	\quad \mbox{for all}\,\,  k \ge k_0.
	\]
	Moreover, the hierarchy of relaxations
	\reff{3.3}-\reff{d3.3} is also tight.
\end{theorem}
\begin{proof}
	The proof is almost the same as for Theorem~\ref{thm3.12}.
	Under the given assumptions,
	there exist $\sigma \in \qmod{\tilde{c}_{\mathcal{I}^{*}}}$,
	$\phi \in \ideal{ \tilde{c}_{\mathcal{E}} }$ such that
	\[
	\tilde{f}-f_{\min}x_0^d = \phi + \sigma.
	\]
	So, the hierarchy of \reff{hmg:even:SOS}-\reff{hmg:even:MOM} is tight.
	Since $\qmod{\tilde{c}_{\mathcal{I}^{*}}} \subseteq \qmod{\tilde{c}_{\mathcal{I}}}$,
	the hierarchy of \reff{3.3}-\reff{d3.3} is also tight.
\end{proof}

In the following, we discuss optimality conditions for
minimizers  at infinity of \reff{hom:dis}.
Suppose $\tilde{x}^* = (0, x^*)$
is a minimizer at infinity for \reff{hom:dis}.
Let $J_1(x^*)$ be as in \reff{jx}.
Since $d$ is even,
the KKT equation for \reff{hom:dis} at $\tilde{x}^*$ is
\be \label{KKT:inf:even}
\left[\begin{array}{c}
	f^\se(x^*)\\ 	\nabla f^\hm(x^*)
\end{array} \right]
=\sum\limits_{i \in J_1(x^*)} \lambda_i
\left[\begin{array}{c}
	c_{i}^\se\left(x^*\right) \\	\nabla c_{i}^\hm(x^*)
\end{array}\right] +
\bar{\lambda} \bbm 	0 \\ x^* \ebm.
\ee
The following is a similar version of Theorem~\ref{p4.7}.

\begin{theorem}  	 \label{prop:even}
	Suppose $K$ is closed at $\infty$ and  $\tilde{x}^* = (0, x^*)$
	is a minimizer at $\infty$ for \reff{hom:dis}.
	Assume $f$ and $c_j$ $(j\in \mathcal{I})$ all have even degrees.
	Let $d = \deg(f) > 1$ and let $J_1(x^*)$ be as in \reff{jx}.
	Then, we have:
	
	\bit
	
	\item [(i)] The LICQC holds for \reff{hom:dis} at $\tilde{x}^*$
	if and only if the gradients
	\be  \nn
	\bbm
	c_{i}^\se(x^*) \\
	\nabla c_{i}^\hm (x^*)
	\ebm\,\,  \big( i\in J_1(x^*) \big)
	\ee
	are linearly independent.

	\item [(ii)] Suppose the LICQC holds  for \reff{hom:dis} at $\tilde{x}^*$.
	Then, $\bar{\lmd}=0$ and \reff{KKT:inf:even} is reduced to
	\be \label{FOOC:even:inf}
	\bbm
	f^\se(x^*)	\\
	\nabla f^\hm(x^*)
	\ebm =
	\sum_{i \in J_1(x^*)} \lambda_i
	\bbm
	c_{i}^\se\left(x^*\right)\\
	\nabla c_{i}^\hm(x^*)	
	\ebm.
	\ee
	The SCC holds  for \reff{hom:dis} at $\tilde{x}^*$
	if and only if $\lambda_i>0$ for $i \in  J_1(x^*)\cap \mathcal{I}.$

	\item [(iii)] Suppose the LICQC holds for \reff{hom:dis} at $\tilde{x}^*$.
	Let $\lmd_i$ be Lagrange multipliers as in the item (ii)
	and let $H$ be the matrix
	\begin{eqnarray}  \label{Hes:inf:even}
		H \, =
		\bbm
		2 f^\thi(x^*)-d(d-1)f_{\min} 0^{d-2}  &  \nabla f^\se(x^*)^\mathrm{T} \\
		\nabla f^\se(x^*) &   \nabla^2 f^\hm(x^*)
		\ebm  \qquad  \\
		- \sum\limits_{i \in J_1(x^*)} \lambda_i
		\bbm
		2 c_i^\thi(x^*)  &  \nabla c_i^\se(x^*)^\mathrm{T} \\
		\nabla c_i^\se(x^*) &   \nabla^2 c_i^\hm(x^*)
		\ebm .  \nn
	\end{eqnarray}
	Then, the SOSC holds for \reff{hom:dis} at $\tilde{x}^*$
	if and only if $\tilde{y}^\mathrm{T}H\tilde{y} > 0$
	for every nonzero $\tilde{y} = (y_0, y)$ satisfying
	\be \label{hom:sed}
	\boxed{
		\baray{c}
		y_0c_{i}^\se\left(x^*\right)+y^{\mathrm{T}}
		\nabla c_{i}^\hm(x^*)=0 ~(i \in J_1(x^*)),\\
		y^{\mathrm{T}}x^*=0
		\earay
	}.
	\ee

	\eit
\end{theorem}
\begin{proof}
	The proofs for the item (i) and  (ii) are exactly the same as for
	(i), (ii) in Theorem \ref{p4.7}.
	So they are omitted for cleanness.
	The item (iii) can be shown as follows.
	The matrix $H$ is the Hessian at $\tilde{x}^*$of
	the Lagrange function
	\[
	L(\tilde{x}) =  \tilde{f}(\tilde{x}) -
	\sum_{i \in J_1(x^*)} \lambda_i \tilde{c}_i(x).
	\]
	A vector $\tilde{y}=(y_0,y)$ lies in the tangent space
	of active constraints of \reff{hom:dis} at $\tilde{x}^*$
	if and only if it satisfies \reff{hom:sed}.
	Thus, the SOSC holds for \reff{hom:dis} at $\tilde{x}^*$
	if and only if  $\tilde{y}^\mathrm{T}H\tilde{y} > 0$
	for every nonzero $\tilde{y}$ satisfying \reff{hom:sed}.
\end{proof}

\section{Extensions of Putinar-Vasilescu's Positivstellensatz}
\label{sc:PV}

In this section,  we generalize the Putinar-Vasilescu's Positivstellensatz
\cite{putinar1999positive,putinar1999solving}
to polynomials that are nonnegative on unbounded semialgebraic sets.
Under some assumptions on optimality conditions,
we prove the conclusions of the Putinar-Vasilescu's Positivstellensatz.

For a polynomial tuple $g \coloneqq (g_1,\dots,g_m)$,
consider the semialgebraic set
\be  \label{5.1}
S \,  \coloneqq  \, \{x\in\mathbb{R}^n\mid g_j(x) \geq 0,~ j=1,\dots,m\} .
\ee
Recall that $f^\hm$ denotes the homogeneous part of the highest degree for $f$.
The following is the classical Putinar-Vasilescu's Positivstellensatz.

\begin{theorem} (\cite[Theorems~1,2]{putinar1999positive}) \label{thm5.1}
	Let $f,g_1,\dots ,g_m \in \mathbb{R}[x]$ and $S$ be
	as in $(\ref{5.1})$. Then,  we have:
	\bit
	
	\item [(i)]	
	Suppose $f$, $g_1,\dots ,g_m$ are homogeneous polynomials of even degrees.
	If $f >0$ on $S \backslash \{0\}$, then
	$\|x\|^{2k}f \in \qmod{g}$ for some power $k \in \mathbb{N}$.
	
	\item [(ii)]
	If the form $f^\hm$ is positive definite in $\re^n$ and $f >0$ on $S$, then
	$(1+\|x\|^2)^kf \in \qmod{g}$ for some power $k \in \mathbb{N}$.
	
	\eit
\end{theorem}

For the case that $S = \re^n$, if $f$ is a positive definite form, then
$\|x\|^{2k}f$ is SOS for some power $k$.
This conclusion is referred to as the
Reznick's Positivstellensatz and it is shown in \cite{reznick2000some}.

First, we generalize the item (i) of Theorem~\ref{5.1}.
Consider the normalized optimization problem
\be  \label{5.2}
\left\{ \baray{rl}
\min & f(x) \\
\st  & g_{j}(x)\geq 0,~j=1,\ldots,m, \\
& \|x\|^2-1=0.
\earay \right.
\ee

\begin{theorem} 	\label{lemma3.14}
	Let $S$ be the set as in $(\ref{5.1})$.
	Suppose $f, g_1,\dots ,g_m$ are homogeneous polynomials of even degrees
	such that $f \ge 0$ on $S$. If
	the LICQC, SCC and SOSC hold at every minimizer of \reff{5.2}, then
	$\|x\|^{2k} f \in \qmod{g}$ for some $k \in \mathbb{N}$.
\end{theorem}
\begin{proof}
	
	By Theorem~1.1 of \cite{nieopcd}, since $f_{\min} \ge 0$, there exist polynomials
	$\sigma_i \in \Sigma[x]$ and $h \in \mathbb{R}[x]$
	such that (let $g_0 \coloneqq 1$)
	\[
	f =  \sum\limits_{i=0}^{m} \sigma_{i}g_i +h \cdot(\|x\|^2-1) .
	\]
	Let $2d_0 \coloneqq \deg(f)$ and $2d_i \coloneqq \deg(g_i)$.	
	In the above identity,
	if we substitute $x_i$ for $\frac{x_i}{\|x\|}$, then we get
	\be  \nn
	\frac{f(x)}{\|x\|^{2d_0}}=  \sum\limits_{i=0}^{m}
	\sigma_{i}\left(\frac{x}{\|x\|}\right)\frac{g_i(x)}{\|x\|^{2d_i}} .
	\ee
	Furthermore, we can also get
	\be  \nn
	\frac{2f(x)}{\|x\|^{2d_0}}=  \sum\limits_{i=0}^{m}
	\Big[\sigma_{i}\left(\frac{x}{\|x\|}\right) +
	\sigma_{i}\left(\frac{-x}{\|x\|}\right)\Big]
	\frac{g_i(x)}{\|x\|^{2d_i}} .
	\ee
	Note that the odd degree terms in $\sigma_{i}\left(\frac{x}{\|x\|}\right) +
	\sigma_{i}\left(\frac{-x}{\|x\|}\right)$ are cancelled.
	The above implies that $\|x\|^{2k}f \in \qmod{g}$
	when $k$ is large enough.
\end{proof}

Second, we generalize the item (ii) of Theorem~\ref{5.1}.
Consider the optimization
\be  \label{5.3}
\left\{ \baray{rl}
\min   & f(x) \\
\st & g_{j}(x)\geq 0,~j=1,\dots,m. \\		
\earay \right.
\ee
For convenience, we still let
$f_{\min}$ denote the minimum value of \reff{5.3}.
The homogenized optimization problem is
\be   \label{5.4}
\left\{ \baray{rl}
\min &  \tilde{f}(\tilde{x}) - f_{\min} \cdot (x_0)^d \\
\st   &  \tilde{g}_{1}(\tilde{x}) \geq 0, \ldots, \tilde{g}_{m}(\tilde{x}) \geq 0,\\
& \|\tilde{x}\|^2-1=0.
\earay \right.
\ee

\begin{theorem}  \label{thm5.5}
	Let $S$ be as in \reff{5.1}. Suppose  $f \geq0$ on $S$, then we have:
	
	\bit
	
	\item [(i)]
	Suppose $S$ is closed at $\infty$ and the degrees of $f,g_1,\dots ,g_m$ are even. If the LICQC, SCC and SOSC hold at every minimizer of \reff{5.3}, including the one at infinity,
	then $(1+\|x\|^2)^{k} f \in \qmod{g}$ for some $k \in \mathbb{N}$.
	
	\item [(ii)]
	If the form $f^\hm$ is positive definite in $\re^n$ and
	the LICQC, SCC and SOSC hold at every minimizer of \reff{5.3},
	then $(1+\|x\|^2)^{k}f \in \qmod{g}$ for some $k \in \mathbb{N}$.
	
	\eit
\end{theorem}
\begin{proof}
	(i) It follows from Theorem~\ref{lemma3.11} that the LICQC, SCC and SOSC
	hold at every  regular minimizer of \reff{5.4}. By Theorem~\ref{thm:ficvg:even}, since $f_{\min} \ge 0$,
	there exist polynomials $\sigma_i \in \Sigma[\tilde{x}]$
	and $h \in \mathbb{R}[\tilde{x}]$ such that (note $g_0 = 1$)
	\be   	\label{eq3.9}
	\tilde{f}  =  \sum\limits_{i=0}^{m}
	\sigma_{i}\tilde{g}_i +h \cdot(\|\tilde{x}\|^2-1).
	\ee
	As in the proof of Theorem $\ref{lemma3.14}$, we can similarly show that
	\[
	\|\tilde{x}\|^{2k}\tilde{f} \in \qmod{ \tilde{g}_1,\dots,\tilde{g}_m },
	\]
	for some $k \in \mathbb{N}$. Substituting $x_0$ for $1$,
	we get $(1+\|x\|^2)^{k}f \in \qmod{g}$.

	\smallskip \noindent
	(ii) For each $i$, let
	$ \theta_i  \coloneqq  2\lceil \frac{\deg(g_i)}{2}\rceil-\deg(g_i).$
	We consider the following homogenized optimization problem
	\be 	\label{5.6}
	\left\{ \begin{array}{rl}
		\min\limits_{ \tilde{x} \in \re^{n+1} }  &
		\tilde{f}( \tilde{x} )  - f_{\min} \cdot (x_0)^d  \\
		\st  & x_0^{\theta_1}\tilde{g}_{1}(\tilde{x}) \geq 0, \ldots,
		x_0^{\theta_m}\tilde{g}_{m}(\tilde{x}) \geq 0, \\
		&    \| \tilde{x} \|^2 -1 = 0.
	\end{array} \right.
	\ee
	Note that $f_{\min} \ge 0$ and the degree of $f$ is even. Let
	$\tilde{u}=(u_0,u)$ be a feasible point of \reff{5.6}.
	If $u_0=0$, then $\|u\|=1$ and
	$\tilde{f}(\tilde{u}) - f_{\min} \cdot (u_0)^d =f^\hm(u)>0$,
	since $f^\hm$ is a positive definite form.  If $u_0\neq 0$,
	then $u/u_0$ is feasible for \reff{5.3} and
	\[
	\tilde{f}( \tilde{u} )  - f_{\min} \cdot (u_0)^d =
	u_0^d \big( f(u/u_0) - f_{\min} \big) \geq 0 .
	\]
	Hence, the optimal value of \reff{5.6} is zero.
	By similar arguments as for Theorem~\ref{lemma3.11}, the LICQC, SCC and SOSC
	hold at every  regular minimizer of \reff{5.6}. Since $f^\hm$ is a positive definite form,
	\reff{5.6} has no minimizers at infinity.
	Note that each $x_0^{\theta_i}\tilde{g}_{i}$ is a polynomial of even degree.
	As for part (i), we can similarly show that
	$(1+\|x\|^2)^{k}f \in \qmod{g}$ for some $k \in \mathbb{N}$.
\end{proof}

When one of the LICQC, SCC and SOSC fails to hold, the conclusion
$(1+\|x\|^2)^{k}f \in \qmod{g}$ may not hold.
We refer to \cite{MLM21} for such examples.
When there is an equality constraint in \reff{1.1},
we cannot simply replace $c_i(x)=0$ by two inequalities
$c_i(x)\geq 0$ and $-c_i(x)\geq0$,
since the LICQC always fails for the latter case.
For the case of equality constrains,
we need to assume that $\ideal{\tilde{c}_{\mathcal{E}}}$
is real radical.
For convenience of notation, denote the polynomial tuples
\be \nonumber
c_{\mc{E}}  \coloneqq  \big\{ c_i(x) \big\}_{i\in \mc{E} }, \quad
c_{\mc{I}}  \coloneqq  \big\{ c_j(x) \big\}_{j \in \mc{I} }.
\ee The following is the Positivstellensatz
for sets with equality constraints.

\begin{thm}  	\label{thm5.6}
	Let $K$ be the feasible set of \reff{1.1}.
	Assume that $\ideal{\tilde{c}_{\mathcal{E}}}$ is real radical
	and $f \ge 0$ on $K$, then we have:
	\bit
	
	\item [(i)] Suppose $K$ is closed at $\infty$
	and the degrees of $f,c_{i}~(i \in \mathcal{I} )$ are all even.
	If the LICQC, SCC and SOSC hold at every minimizer
	of \reff{1.1}, including the one at infinity,
	then $(1+\|x\|^2)^{k}f \in \ideal{c_{\mc{E}}}+ \qmod{c_{\mc{I}}}$
	for some $k \in \mathbb{N}$.

	\item [(ii)]
	Suppose the form $f^\hm$ is positive definite in $\re^n$.
	If the LICQC, SCC and SOSC hold at every minimizer of \reff{1.1}, then
	$(1+\|x\|^2)^{k}f \in \ideal{c_{\mc{E}}} + \qmod{ c_{\mc{I}} }$
	for some $k \in \mathbb{N}$.
	
	\eit
\end{thm}

\begin{proof}
	(i) Since the degrees of $f,c_{i}~(i \in \mathcal{I})$ are all even,  we
	consider the optimization problem \reff{hom:dis}.
	The conclusion can be shown in the same way
	as for item (i) of Theorem~ \ref{thm5.5}.

	\smallskip \noindent
	(ii) For each $j \in \mathcal{I} $, let
	$ \theta_j  \coloneqq  2\lceil \frac{\deg(c_j)}{2}\rceil-\deg(c_j).$
	Consider the following homogenized optimization problem
	\be \label{opt:PV:eq}
	\left\{ \baray{rl}
	\min &  \tilde{f}(\tilde{x})  - f_{\min} \cdot (x_0)^d \\
	\st &   \tilde{c}_{i}(\tilde{x})=0~(i \in \mathcal{E}),\\
	& x_0^{\theta_j} \tilde{c}_{j}(\tilde{x}) \geq 0~(j \in \mathcal{I}),\\
	&  \|\tilde{x}\|^2-1=0.
	\earay \right.
	\ee
	Since $f^\hm$ is a positive definite form,
	the optimization \reff{opt:PV:eq} has no minimizers at infinity.
	By similar arguments as for Theorem~\ref{lemma3.11},  the optimality conditions LICQC, SCC and SOSC
	hold at every regular minimizer of \reff{opt:PV:eq}.
	The conclusion can be shown in the same way
	as for item (ii) of Theorem~ \ref{thm5.5}.
\end{proof}

\section{The Moment-SOS hierarchy with denominators}
\label{sc:denom}

The Putinar-Vasilescu's Positivstellensatz motivates
Moment-SOS relaxations with denominators for solving polynomial optimization.
In view of Theorem~\ref{thm5.6},
we consider the hierarchy of relaxations ($d=\deg(f)$)
\be  \label{rel2}
\left\{\baray{rl}
\max  &   \gamma \\
\st  &  (1+\|x\|^2)^{k-\lceil \frac{d}{2}\rceil} (f-\gamma) \in
\ideal{c_{\mathcal{E}} }_{2k} + \qmod{ c_{\mathcal{I}} }_{2k} ,
\earay \right.
\ee
for the order $k \ge \lceil \frac{d}{2}\rceil$.
The relaxation \reff{rel2} is essentially expressing
$f-\gamma$ in terms of sums of squares of rational polynomials,
with the denominator a power of $1+\|x\|^2$.
Let $f_k^\den$ denote the optimal value of \reff{rel2}
for the relaxation order $k$.
The following is the comparisons between the relaxations
\reff{3.3} (or \reff{hmg:even:SOS} for the even degree case)
and \reff{rel2}:
\bit

\item The polynomials in \reff{rel2} are only in
the variable $x$, while the polynomials in
\reff{3.3} and \reff{hmg:even:SOS}
are in both $x$ and $x_0$.

\item The relaxation~\reff{3.3} uses the constraint
$x_0 \ge 0$, while \reff{rel2} and \reff{hmg:even:SOS}
do not use it. So \reff{3.3} is the strongest among them.

\item When the degrees of $f$ and $c_i$ ($i \in \mc{I}$)
are all even, the relaxation \reff{hmg:even:SOS}
is equivalent to \reff{rel2}.
This can be observed as in the proof of
Theorem~\ref{lemma3.14}. For the special case that
$f$ is an even degree form and $K$ is the unit sphere,
the equivalence is shown in \cite{dlp}.

\eit

We remark that \reff{rel2} is the relaxation given in \cite{MLM21}
for the parameter $\epsilon=0$.
It is conjectured in \cite[Sec.~4.2]{MLM21}
that the hierarchy of \reff{rel2} has finite convergence,
under some optimality conditions.
We prove this conjecture is true
under the assumptions of Theorem~\ref{thm5.6}.
Recall that for the even degree case,
the LICQC, SCC and SOSC are said to hold at a minimizer  at infinity $x^*$
if they hold for \reff{hom:dis} at $(0,x^*)$.

\begin{thm}  \label{thm:conj:MLM21}
	Let $K$ be the feasible set of \reff{1.1}.
	Assume $\ideal{\tilde{c}_{\mathcal{E}}}$ is real radical.
	
	\bit
	
	\item [(i)] Suppose $K$ is closed at $\infty$
	and the degrees of $f,c_{i}~(i \in \mathcal{I} )$ are even.
	If the LICQC, SCC and SOSC hold at every  minimizer of \reff{1.1}, including the one at infinity,
	the hierarchy of \reff{rel2} has finite convergence, i.e.,
	$f_k^\den = f_{\min}$ for all $k$ big enough.

	\item [(ii)]
	Suppose $f^\hm$ is a positive definite form in $\re^n$.
	If the LICQC, SCC and SOSC hold at  every  minimizer of \reff{1.1},
	then $f_k^\den = f_{\min}$ for all $k$ big enough.
	
	\eit
\end{thm}
\begin{proof}
	The conclusions follow from Theorem~\ref{thm5.6}.
	For the item (ii), there are no minimizers at infinity
	if $f^\hm$ is a positive definite form in $\re^n$.
\end{proof}

In Theorem~\ref{thm:conj:MLM21}(i),
the degrees of $f,c_{i}~(i \in \mathcal{I} )$
are assumed to be even, but the degrees of
equality constraining polynomials
$c_{i}~(i \in \mathcal{E} )$ can be either odd or even.

A special case of \reff{1.1} is that there are no constraints.
Then the resulting version of the relaxation \reff{rel2} is
\be  \label{sos:den:Reznic}
\left\{\baray{rl}
\max  &   \gamma \\
\st  &  (1+\|x\|^2)^{k-\lceil \frac{d}{2}\rceil} (f-\gamma)
\in \Sig[x]_{2k},
\earay \right.
\ee
for the order $k \ge \lceil \frac{d}{2}\rceil$. The degree $d$ must be even
for $f_{\min} > - \infty$, when there are no constraints.
When $f^\hm$ is a positive definite form,
the asymptotic convergence  of \reff{sos:den:Reznic}
can be shown by Reznick's Positivstellensatz
\cite[Sec.~7]{reznick2000some}.
We have the following theorem about the finite convergence.

\begin{thm}  \label{thm:uncon}
	Suppose 
	$K=\mR^n$ and $f_{\min} > - \infty$.
	Let $f_k^\den$ denote the optimal value of
	\reff{sos:den:Reznic} for the order $k$.
	
	\bit
	
	\item [(i)]
	If the SOSC holds at every minimizer of \reff{1.1}, including the one at infinity, then
	$f_k^\den = f_{\min}$ for all $k$ big enough.

	\item [(ii)]
	If $n \le 2$,  then $f_k^\den = f_{\min}$ for all $k$ big enough.
	
	\eit
\end{thm}
\begin{proof}
	The item (i) follows from Theorem~\ref{thm5.6} (i).
	It is shown by Scheiderer~\cite{Scheiderer2006} that if
	$p(u)$ is a nonnegative form in $u \in \re^n$
	with $n \le 3$, then $\|u \|^{2N} p(u)$ is SOS
	when $N$ is big enough. This implies that
	\[
	(1+\|x\|^2)^{k-\lceil \frac{d}{2}\rceil} ( f - f_{\min} )
	\in \Sig[x]_{2k}
	\]
	when $k$ is big enough, for the case $n \le 2$.
	The item (ii) follows from this conclusion.
\end{proof}

\section{Numerical examples}
\label{num:ex}
This section presents some examples on  the relaxations $(\ref{3.3})$  for solving the
optimization problem $(\ref{1.1})$. The computation is implemented in
MATLAB R2019a, on a Dell  Desktop with CPU@2.90GHz and RAM 32.0G. The
relaxations $(\ref{3.3})-(\ref{d3.3})$   are solved by the software {\tt GloptiPoly~3} \cite{2009GloptiPoly}, which calls
the SDP package {\tt SeDuMi} \cite{sturmusing}.
For neatness, only four decimal digits are displayed
for computational results. The feasible sets of all examples are unbounded.

\subsection{The case with minimizers}

A convenient criterion for obtaining minimizers
is the flat extension or truncation. Denote the degree
\begin{equation}
	d_K  \coloneqq \max\{ \lceil \deg(f)/2 \rceil,
	\lceil \deg(c_i)/2 \rceil
	(i \in \mathcal{E} \cup \mc{I}) \}.
\end{equation}
Suppose $y^{*}$ is a minimizer of $(\ref{d3.3})$ for the relaxation order $k$.
If there exists an integer $t \in[d_K, k]$ such that

\be \label{FT:y*}
\operatorname{rank} M_{t} [ y^* ] \, = \,
\operatorname{rank} M_{t-d_K}[ y^* ],
\ee
then we can get one or several minimizers for \reff{3.5} (see \cite{CF05,HenLas05,Lau05,nie2013certifying}).
When \reff{FT:y*} holds, we have the decomposition
\[
y^*|_{2t} = a_1 [\bpm \tau_1 \\ v_1 \epm]_{2t}
+ \cdots + a_r [\bpm \tau_r \\ v_r \epm]_{2t}
\]
for positive scalars $a_i > 0$
and distinct points $(\tau_i, v_i) \in \widetilde{K}$.
Denote two label sets
\[
I_1 = \{ i : \, \tau_i > 0\}, \quad
I_2 = \{ i : \, \tau_i = 0\}.
\]
For each $i \in I_1$, let $u_i = v_i/\tau_i$.
Then $u_i \in K$ for each $i \in I_1$
and $v_i \in K^{(1)}$ for each $i \in I_2$.
Let $\nu_i = a_i (\tau_i)^d$ for each $i \in I_1$
and $\nu_i = a_i  $ for each $i \in I_2$.

\begin{lem}  \label{lm:FT:min}
	Suppose $y^{*}$ is a minimizer of $(\ref{d3.3})$
	and the rank condition $\reff{FT:y*}$ is satisfied.
	Let each $u_i, v_i, \nu_i$ be as above.
	Then the set $I_1 \ne \emptyset$ and
	each $u_i$ ($i \in I_1$) is a minimizer for \reff{1.1},
	and each $v_i$ ($i \in I_2$) is a minimizer at infinity
	for \reff{1.1}.
\end{lem}
\begin{proof}
	Note that $u_i \in K$ for each $i \in I_1$
	and $v_i \in K^{(1)}$ for each $i \in I_2$,
	the constraint $\langle x_0^d, y \rangle = 1$
	in \reff{d3.3} implies that
	$
	\sum_{ i \in I_1 } \nu_i   =  1.
	$
	So at least one $\nu_i >0$ and $I_1 \ne \emptyset$. Since $v_i$ $(i \in I_2)$
	is a feasible point of $K^\hm$, by item (iii) of Theorem \ref{thm:prop:inf},
	we have $f^{(1)}(v_i)\geq 0$ for  $i \in I_2$ and hence
	\[
	\begin{split}
		f_{\min} & =\langle \tilde{f}, y^* \rangle =
		\sum_{ i \in I_1 } \nu_i f(u_i) +
		\sum_{ i \in I_2 } \nu_i f^{(1)}(v_i)\\
		&\geq \sum_{ i \in I_1 } \nu_i f(u_i)
		\geq \sum_{ i \in I_1 } \nu_i f_{\min}=f_{\min}.
	\end{split}
	\]
	Therefore, $f(u_i)=f_{\min}$ for $i \in I_1$,
	and $f(v_i)=0$ for $i \in I_2$.
\end{proof}

We refer to \cite{CF05,HenLas05,Lau05,nie2013certifying}
for flat extensions and truncations.
The procedure of extracting minimizers by using \reff{FT:y*}
is implemented in
{\tt GloptiPoly~3} \cite{2009GloptiPoly}.
The rank condition \reff{FT:y*} is a sufficient
and almost necessary condition
for checking convergence of the Moment-SOS hierarchy
\cite{nie2013certifying}.

\begin{exm}
	(i) Consider the optimization (the variable $x_0 = 1$):
	\[
	\min\limits_{x \in \re^4} \quad \sum_{i=0}^{4}
	\prod_{i \ne j \in \{0, 1, \ldots, 4\} }\left(x_{i}-x_{j}\right)
	+ 0.1 \Big( \sum_{i=1}^{4}  x_i^4 \Big).
	\]
	The first part  of the objective is a nonnegative
	but non-SOS polynomial \cite{reznick2000some}.
	For the order $k=3$, we get $f_3 \approx  0.0763$ and the minimizer
	$0.5757 \cdot (1,1,1,1).$
	There are no minimizers at infinity.

	\noindent
	(ii) Consider the optimization:
	\[
	\left\{ \baray{rl}
	\min\limits_{x \in \re^2} &
	x_1^2x_2 + x_2^2x_1 -3x_1x_2  \\
	\st  &  x_1 \ge 0, \, x_2 \ge 0.
	\earay \right.
	\]
	Up to a constant, the objective becomes the dehomogenization of
	the Motzkin form \cite{reznick2000some}
	if each $x_i$ is changed to $x_i^2$. For the order $k=3$,
	we get $f_3 \approx -1.0000$ and a  minimizer
	$(1.0000, 1.0000)$. We also get two minimizers at infinity.
	They are $(1.0000, 0.0000)$ and $(0.0000, 1.0000)$.
	
	\noindent
	(iii) Consider the optimization:
	\[
	\left\{ \baray{rl}
	\min\limits_{x \in \re^2}  &  x_1^2x_2 + x_2^2 + x_1 -3x_1x_2   \\
	\st  &  x_1 \ge 0, \, x_2 \ge 0.
	\earay \right.
	\]
	The objective becomes the dehomogenization of the Choi-Lam form \cite{reznick2000some}
	if each $x_i$ is changed to $x_i^2$. For the order $k=2$,
	we get $f_2 \approx 2.9586\times 10^{-8}$ and two  minimizers:
	$(1.0000, 1.0000)$, $(0.0000, 0.0000)$.
	We also get two minimizers at infinity.
	They are $(1.0000, 0.0000)$ and $(0.0000, 1.0000)$.

	\noindent
	(iv) Consider the optimization:
	\[
	\left\{ \baray{rl}
	\min\limits_{x \in \re^2} &
	x_1^3+x_2^3+3x_1x_2-x_1^2(x_2+1)-x_2^2(x_1+1)-(x_1+x_2)  \\
	\st  &  x_1 \ge 0, \, x_2 \ge 0.
	\earay \right.
	\]
	Up to a constant, the objective becomes the dehomogenization of the Robinson form \cite{reznick2000some}
	if each $x_i$ is changed to $x_i^2$. For the order $k=2$,
	we get $f_2 \approx -1.0000$ and three  minimizers:
	\[
	(1.0000, 1.0000), \quad (0.0000, 1.0000), \quad (1.0000, 0.0000).
	\]
	We also get one minimizer at infinity: $(0.7071, 0.7071)$.
\end{exm}

\begin{exm}
	(i) Consider the constrained optimization:
	\begin{equation*}
		\left\{ \baray{rl}
		\min\limits_{x \in \mathbb{R}^{2}} & x_{1}^{6}+x_{2}^{6}+1+3 x_{1}^{2} x_{2}^{2}
		-x_{1}^{2}\left(x_{2}^{4}+1\right)  \\
		& \qquad -x_{2}^{2}\left(1+x_{1}^{4}\right)
		-\left(x_{1}^{4}+x_{2}^{4}\right)  \\
		\st & x_{1}+x_{2}+1=0
		\earay \right.
	\end{equation*}
	This problem is a variation of  Example~5.2 in \cite{nie2013exact}.
	For the order $k=3$, we get $f_3 \approx 5.4436\times10^{-6}$
	and we get two  minimizers:
	$( -1.0000, 0.0000)$, $( 0.0000, -1.0000)$.
	We also get two minimizers at infinity:
	\[
	(0.7068, -0.7074), \quad (-0.7074, 0.7068).
	\]
	
	\noindent
	(ii) Consider the constrained optimization:
	\begin{equation*}
		\left\{ \baray{rl}
		\min\limits_{x \in \re^2}  &  x_1^2+x_2^2  \\
		\st & x_2^2-1 \geq 0, \\
		& x_1^2-2x_1x_2-1 \geq 0, \\
		& x_1^2+2x_1x_2-1 \geq 0.
		\earay \right.
	\end{equation*}
	This example is from \cite{nie2013exact}.
	The minimum value $f_{\min} \approx 6.8284$
	and the minimizers are $(\pm (1+\sqrt{2}), \pm 1)$.
	There are no minimizers at infinity.
	For the order $k=3$, we get $f_3 \approx 6.8284$
	and four minimizers:
	$
	(\pm 2.4142,  \pm 1.0000 ).
	$

	\noindent
	(iii) Consider the constrained optimization:
	\begin{equation*}
		\left\{\baray{rl}
		\min\limits_{x \in \mathbb{R}^{3}}  &   x_{1}^{2}\left(x_{1}-1\right)^{2}+x_{2}^{2}\left(x_{2}-1\right)^{2}+
		x_{3}^{2}\left(x_{3}-1\right)^{2}  \\
		& \qquad +2 x_{1} x_{2} x_{3}\left(x_{1}+x_{2}+x_{3}-2 \right)\\	
		& \qquad +\left(x_{1}-1\right)^{2}+\left(x_{2}-1\right)^{2}+\left(x_{3}-1\right)^{2} \\		
		\st & x_{1}-2x_{2}^2 \geq 0, \  x_{2}-x_{3} \geq 0 .	
		\earay \right.
	\end{equation*}
	The sum of the first four
	terms of the objective is a nonnegative polynomial \cite{reznick2000some}.
	For $k=2$, we get $f_2 \approx 0.4708$ and the minimizer
	$(0.6979, 0.6980,0.6978)$.
	
	\noindent
	(iv) Consider the constrained optimization:
	\begin{equation*}
		\left\{ \baray{rl}
		\min\limits_{x \in \mathbb{R}^{2}}  &
		2x_{1}^{3}+2x_{2}^{3}+4 x_{1} x_{2} -x_{1}\left(x_{2}^{2}+1\right)  \\
		& \qquad +x_{2}\left(1+x_{1}^{2}\right)+x_{1}^{2}+x_{2}^{2}\\		
		\st &  x_{1} \geq 1,  \, x_{2}  \geq 1.	
		\earay \right.
	\end{equation*}
	For $k=2$, we get $f_2 \approx 2.0000$ and the minimizer
	$(1.0000, 1.0000)$.
\end{exm}

\subsection{The case with no minimizers}
\label{ssc:notattain}

When the minimum value $f_{\min}$ of \reff{1.1} is not achievable,
i.e., \reff{1.1} has no minimizers,
then the moment relaxation \reff{d3.3}
can not have a minimizer $y^*$ satisfying \reff{FT:y*}.
This is implied by Lemma~\ref{lm:FT:min}.
Indeed, the moment relaxation \reff{d3.3}
typically does not achieve its optimal value,
i.e., it does not have optimizers either.
For such cases, there often exist numerical issues
for solving Moment-SOS relaxations \reff{3.3}-\reff{d3.3},
although the finite convergence is guaranteed
under some assumptions on minimizers at infinity.

For instance, consider the unconstrained optimization with the objective
$
f = x_1^4 + (x_1x_2-1)^2.
$
Clearly, $ f_{\min} = 0$ is not achievable and
$f_k =f^{\prime}_k=0$ for all $k \ge 2$.
The moment relaxation \reff{d3.3} does not have an optimizer.
The objective in \reff{d3.3} is
$
y_{040} + y_{022} -2 y_{211} + y_{400}.
$
For $k\geq 2$, the moment matrix constraint $M_k[y] \succeq 0$ implies that
\[
\left(
\begin{array}{rccccc}
	y_{400} & y_{310} & y_{301} & y_{220} & y_{211} & y_{202} \\
	y_{310} & y_{220} & y_{211} & y_{130} & y_{121} & y_{112} \\
	y_{301} & y_{211} & y_{202} & y_{121} & y_{112} & y_{103} \\
	y_{220} & y_{130} & y_{121} & y_{040} & y_{031} & y_{022} \\
	y_{211} & y_{121} & y_{112} & y_{031} & y_{022} & y_{013} \\
	y_{202} & y_{112} & y_{103} & y_{022} & y_{013} & y_{004} \\
\end{array}
\right) \succeq 0.
\]
Note that $y_{400}=1$.
When $y$ is feasible for \reff{d3.3}, we can get
\[
y_{040} \geq 0, \quad y_{022} \geq 0, \quad
y_{022} -2 y_{211} + 1 \geq 0.
\]
If $y_{040} >0$, then $\langle \tilde{f}, y \rangle >0$.
If $y_{040} =0$, then
\[
y_{022}= 0, \quad y_{211} =0, \quad
y_{022} -2 y_{211} + 1 = 1 > 0.
\]
The objective of \reff{d3.3} is positive
for all feasible $y$, so it does not achieve the optimal value.
There are numerical troubles for
for solving the Moment-SOS relaxations.

When $f_{\min}$ is not achievable, a more numerically well-posed problem
is to compute minimizers at infinity.
We apply the Moment-SOS relaxations to solve
the optimization problem \reff{p3.7}
for one or several minimizers $x^*$ at infinity.
When the optimality conditions hold at minimizers of \reff{p3.7},
its Moment-SOS hierarchy has finite convergence,
so its minimizers can be obtained
(see \cite{nie2013certifying,nieopcd}).
For $f_{\min} > -\infty$, it is necessary that $f^{(1)}(x^*) = 0$.
This is implied by Theorem~\ref{thm:prop:inf}.

\begin{exm}\label{ex7.4}
	(i) Consider the optimization:
	\begin{equation*}
		\min_{x\in \mR^2} \quad x_2^2+(2x_2^2+2x_1x_2+1)^2.
	\end{equation*}
	The minimum value $f_{\min} = 0$ is not achievable.
	The minimizers at infinity are $(\pm1,0)$, $(\frac{1}{\sqrt{2}},-\frac{1}{\sqrt{2}})$, $(-\frac{1}{\sqrt{2}},\frac{1}{\sqrt{2}})$.
	For the order $k=3$, we get $f_3 \approx  9.8893\times 10^{-9}$
	and four minimizers at infinity:
	\[
	(-1.0000,  0.0001), \,(1.0000, -0.0001), \,
	(-0.7071,  0.7071), \, (0.7071, -0.7071).
	\]

	\noindent
	(ii) Consider the optimization:
	\begin{equation*}
		\min_{x\in \mR^3} \quad  x_1^2+(1-x_1x_2)^2 + L(x),
	\end{equation*}
	where $L=x_1^{4} x_2^{2}+x_2^{4} x_3^{2}+x_3^{4} x_1^{2}-3 x_1^{2} x_2^{2} x_3^{2}$
	is the Choi-Lam form. Clearly, $f_{\min}\geq 0$. For the sequence of
	$x^{(k)}=(\frac{1}{n},n,0)$, we have $f(x^{(k)})=\frac{2}{n^{2}}.$
	As $k\rightarrow 0$, we have $f(x^{(k)})\rightarrow 0$,
	which implies $f_{\min} =0$. However, $f_{\min} = 0$ is not achievable,
	since the polynomial $x_1^2+(1-x_1x_2)^2 $ has no real zeros.
	The objective $f$ is not an SOS, since $f^\hm=L(x)$ is not an SOS.
	The minimizers at infinity are
	\[
	\frac{1}{\sqrt{3}}(\pm1,\pm1,\pm1), (\pm1,0,0),(0,\pm1,0),(0,0,\pm1).
	\]
	For $k=5$, we get $f_5 \approx -1.6413\times 10^{-8}$
	and all the minimizers at infinity.

	\noindent
	(iii)  Consider the optimization:
	\begin{equation*}
		\min_{x\in \mR^3} \quad \eps ((x_3^2+x_1x_3+1)^2+x_3^6)+R(x),
	\end{equation*}
	where $\eps >0$ and $R$ is the dehomogenized Robinson polynomial
	\[
	R(x)=1+x_2^6+x_3^6+3x_2^2x_3^2-(x_2^2+x_3^2)
	-x_2^4(1+x_3^2)-x_3^4(1+x_2^2).
	\]
	Clearly, $f_{\min}\geq 0$, since the objective is
	a sum of two nonnegative polynomials.
	For the sequence of $x^{(k)}=(-\frac{1+n^2}{n},1,\frac{1}{n})$,
	we have $f(x^{(k)})=\frac{1+\epsilon-2n^2+n^4}{n^{6}}.$
	As $k\rightarrow 0$, we have $f(x^{(k)})\rightarrow 0$.
	So, $f_{\min} =0$. However, $f_{\min}$ is not achievable,
	since the polynomial $(x_3^2+x_1x_3+1)^2+x_3^2$ has no real zeros.
	For $\eps > 0$ small enough, $f$ is not an SOS.
	The minimizers at infinity are $(\pm1,0,0)$.
	Let $\eps=1$, for the order $k=3$, we get $f_3 \approx 3.1810\times 10^{-10}$
	and two  minimizers at infinity:
	$(\pm 1.0000, 0.0000, 0.0000)$.

	\noindent
	(iv) Consider the optimization:
	\begin{equation*}
		\left\{ \baray{rl}
		\min\limits_{x\in \mR^5}   &  (x_1+x_2+x_3+x_4x_5)^2
		-4 \Big(x_1x_2+x_2x_3+x_3(x_4x_5-1) \\
		& \qquad   +x_4x_5-1+x_1 \Big)  +(x_1-1)^2+x_4^2 \\
		\st  & x_1\geq0, \, x_2-x_1\geq 0, \,  x_3-x_2\geq0, \\
		& x_4-x_3\geq0, \, x_5-x_4\geq0,x_4x_5\geq1.
		\earay \right.
	\end{equation*}
	The minimum value $f_{\min} = 0$ is not achievable.
	The minimizer at infinity is $(0,0,0,0,1)$.
	For the order $k=2$, we get $f_2 \approx  6.2486\times10^{-9}$
	and the minimizer at infinity
	$(0.0000,0.0000,0.0000,0.0000,1.0000)$.
\end{exm}

When minimizers at infinity are obtained,
we can apply the procedure in \cite{vuitan} to determine the value $f_{\min}$.
With these minimizers at infinity, one can drop some connected components
of the tangency variety that do not converge to minimizers at infinity.
This can save computational expense. For instance,
we consider the polynomial $f = (x_1x_2-1)^{2}+x_2^{2}$
and $K = \re^2$. Clearly, the minimum value $f_{\min}$
is not attainable.  By the procedure in \cite{vuitan},
the tangency variety $\Gamma\left(f, \mathbb{R}^{2}\right)$ is given by the equation:
\[
2\left(-x_1^{3} x_2+x_1 x_2^{3}+x_1^{2}-x_1 x_2-x_2^{2}\right) =  0.
\]
For $R>0$ large enough, the set
$\Gamma\left(f, \mathbb{R}^{2}\right) \cap \{x_1^2+x_2^2\geq R\}$
has eight connected components:
\[
\begin{array}{lll}
	\Gamma_{\pm 1}: & x_1:=t, &
	x_2:=-t+\frac{1}{2} t^{-1}+\frac{5}{8} t^{-3}+\cdots, \\
	\Gamma_{\pm 2}: & x_1:=t, &
	x_2:=t+\frac{1}{2} t^{-1}+\frac{3}{8} t^{-3}+\cdots, \\
	\Gamma_{\pm 3}: & x_1:=t, &
	x_2:=t^{-1}-t^{-3}+\cdots, \\
	\Gamma_{\pm 4}: & x_1:=t^{-1}, &
	x_2:=t+t^{-1}-t^{-3}+\cdots,
\end{array}
\]
for a parameter $t \rightarrow +\infty$ or $t \rightarrow - \infty$.
As in \cite{vuitan}, one can compute the asymptotic value of $f$
in each component $\Gamma_{\pm i}$.
Substituting the parametrization in $f$, we get
\[
\baray{ll}
\left.f\right|_{\Gamma_{\pm 1}} =t^{4}+4 t^{2}+2-\frac{23}{8} t^{-2}+\cdots,  &
\left.f\right|_{\Gamma_{\pm 2}} =t^{4}+2+\frac{5}{8} t^{-2}+\cdots \\
\left.f\right|_{\Gamma_{\pm 3}} =t^{-2}-t^{-4}+t^{-6}+\cdots, &
\left.f\right|_{\Gamma_{\pm 4}} =t^{2}+2-t^{-2}-t^{-4}+\cdots .
\earay
\]
The asymptotic values on these components are
\[
\lambda_{\pm 1} \,= \, \lambda_{\pm 2} \, = \, \lambda_{\pm 4} \,= \, +\infty,
\quad \lambda_{\pm 3}=0.
\]
Therefore,   the minimum value
$
f_{\min } =
\min\limits_{k=1,2,3,4} \lambda_{\pm k} = 0.
$
For this example, the minimizers at infinity are $(\pm 1,0)$ and $(0,\pm 1)$.
In the computation, if minimizers at infinity are obtained,
we do not need to consider the components
$\Gamma_{\pm 1}$, $\Gamma_{\pm 2}$,
since their normalizations
%
%
do not converge to minimizers at infinity.
We refer to \cite{ha2009solving,vui2008global,schweighofer2006global}
for related work for the case that $f_{\min}$ is not achievable.

\section{Conclusions and discussions}
\label{sc:dis}

This paper gives a Moment-SOS hierarchy for polynomial optimization
with unbounded sets, based on homogenization.
We prove this hierarchy has finite convergence
under some assumptions on optimality conditions about minimizers.
We also extend the Putinar-Vasilescu type Positivstellensatz
to polynomials that are  nonnegative on unbounded sets.
The classical Moment-SOS hierarchy
with denominators is also studied.
Moreover, we give a positive answer to a conjecture of
Mai, Lasserre and Magron in their recent work \cite{MLM21}.

To prove the convergence of the hierarchy of \reff{3.3}-\reff{d3.3},
we made the assumption that the ideal
$\ideal{\tilde{c}_{\mathcal{E}}}$ is real radical.
We remark that $\ideal{\tilde{c}_{\mathcal{E}}}$ being real radical
is a general condition (see \cite{JosYu16}).
When $\ideal{\tilde{c}_{\mathcal{E}}}$ is not real radical,
we do not know if the hierarchy of \reff{3.3}-\reff{d3.3}
still has the finite convergence,
under the remaining assumptions as in Theorem~\ref{thm3.12}.
However, we have the following result when
$\ideal{\tilde{c}_{\mathcal{E}}}$ is not real radical.

\begin{theorem} \label{thm:radical}
	Assume $K$ is closed at $\infty$. If the LICQC, SCC and SOSC hold
	at every minimizer of (1.1), including the one at infinity,
	then there exists an integer $k_0 >0$ such that
	\be \label{x02k:finitecvg}
	x_0^{2\ell}( \tilde{f}-( f_{\min}-\epsilon) \cdot x_0^d) \in
	\ideal{\tilde{c}_{\mathcal{E}}} +\qmod{\tilde{c}_{\mathcal{I}}}
	\ee
	for every $\eps >0$ and for all $\ell \geq k_0$.
\end{theorem}
\begin{proof}
	As in Theorem~\ref{thm3.12},
	there exists $\sigma \in \qmod{\tilde{c}_{\mathcal{I}}}$ such that
	\begin{equation*}
		\tilde{f}-f_{\min}x_0^d \, \equiv \, \sigma ~ \bmod ~
		\ideal{V_{\mathbb{R}}(\tilde{c}_{\mathcal{E}})}.
	\end{equation*}
	Let $\bar{f}=\tilde{f}-f_{\min} \cdot x_0^d-\sigma$.
	Then $\bar{f}$ vanishes identically on the variety $V_{\mathbb{R}}(\tilde{c}_{\mathcal{E}})$.
	By the Real Nullstellensatz (see \cite{Lau09}), there exist
	$k_1\in \N$, $\sigma_1\in \Sigma[\tilde{x}]$ such that
	$\bar{f}^{2k_1}+\sigma_1  \in	\ideal{\tilde{c}_{\mathcal{E}}}$.
	Let $\omega >0$ be big enough such that the univariate polynomial
	$s(t):=1+t+ \omega t^{2 k_1}$ is SOS (see \cite{Nie13}).
	For each $\epsilon>0$,  we get
	\[
	(\epsilon x_0^d)^{2k_1} s(\frac{\bar{f}}{\epsilon x_0^d}) \,= \,
	(\epsilon x_0^d)^{2k_1-1}(\bar{f}+\epsilon x_0^d) + \omega \bar{f}^{2k_1}.
	\]
	This implies that
	\[
	x_0^{(2k_1-1)d} \big( \tilde{f}-(f_{\min}-\epsilon)x_0^d-\sigma \big) =
	(\epsilon x_0^d)^{2k_1}s(\frac{\bar{f}}{\epsilon x_0^d}) - \omega \bar{f}^{2k_1}.
	\]
	Since $x_0\in \tilde{c}_{\mathcal{I}}$ and $\sigma_1$ is SOS, we have
	\[
	(\epsilon x_0^d)^{2k_1}s(\frac{\bar{f}}{\epsilon x_0^d}) -\omega \bar{f}^{2k_1}
	\in  \Sigma[\tilde{x}]+ 	\ideal{\tilde{c}_{\mathcal{E}}} \subseteq
	\ideal{\tilde{c}_{\mathcal{E}}} +\qmod{\tilde{c}_{\mathcal{I}}},
	\]
	\[
	x_0((\epsilon x_0^d)^{2k_1}s(\frac{\bar{f}}{\epsilon x_0^d})-\omega \bar{f}^{2k_1})
	\in x_0\Sigma[\tilde{x}]+ 	\ideal{\tilde{c}_{\mathcal{E}}} \subseteq \ideal{\tilde{c}_{\mathcal{E}}} +\qmod{\tilde{c}_{\mathcal{I}}}.
	\]
	Let $k_0=\lceil \frac{(2k_1-1)d}{2} \rceil$. Then, we have
	\[
	x_0^{2k_0}(\tilde{f}-(f_{\min}-\epsilon)x_0^d)=x_0^{2k_0}\sigma+
	x_0^{2k_0-(2k_1-1)d}((\epsilon x_0^d)^{2k_1}
	s(\frac{\bar{f}}{\epsilon x_0^d})-\omega\bar{f}^{2k_1}).
	\]
	It implies that
	\[
	x_0^{2k_0}(\tilde{f}-(f_{\min}-\epsilon)x_0^d) \in \ideal{\tilde{c}_{\mathcal{E}}} +\qmod{\tilde{c}_{\mathcal{I}}}.
	\]
	Note that $k_0$ is independent of $\epsilon > 0$.
	The above implies that \reff{x02k:finitecvg}
	holds for all $\eps >0$ and for all $\ell \geq k_0$.
\end{proof}

For a degree $\ell \geq k_0$,
\reff{x02k:finitecvg} motivates the hierarchy of the following relaxations
\begin{equation}  \label{rel}
	\left\{ \baray{rl}
	\max &   \gamma \\
	\st &  x_0^{2\ell}(\tilde{f}(\tilde{x})- \gamma x_0^d )\in
	\ideal{\tilde{c}_{\mathcal{E}}}_{2k} +\qmod{\tilde{c}_{\mathcal{I}}}_{2k}.
	\earay \right.
\end{equation}
Theorem~\ref{thm:radical} shows that the optimal value of relaxations \reff{rel}
has the finite convergence to $f_{\min}$,
even if $\ideal{\tilde{c}_{\mathcal{I}}}$ is not real radical,
under the remaining assumptions. Because of this result,
we make the following conjecture.

\begin{conj}
	When the ideal $\ideal{\tilde{c}_{\mathcal{E}}}$ is not real radical,
	the hierarchy of relaxations \reff{3.3}-\reff{d3.3} is also tight,
	i.e., $f_k = f_k^\prime = f_{\min}$,
	for all $k$ big enough,
	under the other assumptions of Theorem~\ref{thm3.12}.
\end{conj}

When the minimum value $f_{\min}$ is not achievable,
the Moment-SOS hierarchy of relaxations \reff{3.3}-\reff{d3.3}
also has finite convergence under some assumptions
on optimality conditions for minimizers at infinity.
However, there are numerical issues for solving the hierarchy,
since there are no optimizers satisfying flat truncation
for the moment relaxations.
It is an interesting future work to get numerically stable
Moment-SOS relaxations for computing
$f_{\min}$ when it is not achievable.

\

{ \bf Acknowledgements}
	The authors would like to thank the editors and anonymous referees
	for fruitful comments and suggestions.
	Lei Huang and Ya-Xiang Yuan are partially supported by 
	the National Natural Science Foundation of China (No.~12288201).

%
%



\end{document}